\documentclass[12pt, a4paper]{amsart}
\usepackage[hmargin=3cm, vmargin=2.5cm, includefoot]{geometry}
\usepackage[backref=page,bookmarksopen=true]{hyperref}
\usepackage[latin1]{inputenc}  
\usepackage[T1]{fontenc}       
\usepackage[english]{babel}
\usepackage{amsmath}
\usepackage{amssymb}
\usepackage{amsthm}
\usepackage{latexsym}
\usepackage{mathrsfs}
\usepackage{graphics, graphicx}
\usepackage{enumerate, xspace, paralist}
\usepackage[usenames]{color}
\newtheorem{lem}{Lemma}[section]
\newtheorem{prop}[lem]{Proposition}
\newtheorem{cor}[lem]{Corollary}
\newtheorem{thm}[lem]{Theorem}

\newtheorem{claim}{Claim}

\theoremstyle{definition}

\newcommand{\la}{\langle}
\newcommand{\ra}{\rangle}

\DeclareMathOperator{\centra}{\mathscr{Z}}

\DeclareMathOperator{\Ker}{Ker}
\DeclareMathOperator{\Sym}{Sym}
\DeclareMathOperator{\Aut}{Aut}

\DeclareMathOperator{\Stab}{Stab}
\DeclareMathOperator{\Fix}{Fix}
\DeclareMathOperator{\Ch}{Ch}
\DeclareMathOperator{\proj}{proj}
\DeclareMathOperator{\Res}{Res}

\DeclareMathOperator{\dist}{dist}

\newcommand{\inv}{^{-1}}

\begin{document}

\title[Automorphisms of right-angled buildings]{Automorphism groups of right-angled buildings: \\ Simplicity and local splittings}
\author{Pierre-Emmanuel Caprace}
\address{Universit\'e catholique de Louvain, IRMP, Chemin du Cyclotron 2, 1348 Louvain-la-Neuve, Belgium}
\email{pe.caprace@uclouvain.be}
\thanks{P.-E.C. is an F.R.S.-FNRS Research Associate, supported in part by FNRS grant F.4520.11 and the European Research Council (grant \#278469)}


\date{First draft: October 2012; revised: April 2013}

\begin{abstract}
We show that the group of type-preserving automorphisms of any irreducible semi-regular thick right-angled  building is abstractly simple. When the building is locally finite, this gives a large family of compactly generated (abstractly) simple locally compact groups. Specialising to appropriate cases, we obtain examples of such simple groups that are locally indecomposable, but have locally normal subgroups decomposing non-trivially as direct products, all of whose factors are locally normal.
\end{abstract}

\maketitle


\begin{flushright}
\begin{minipage}[t]{0.7\linewidth}
\small\itshape  {``Everywhere there was evidence of a collective obsession with the comforting logic of right angles.''}
\upshape
\begin{flushright}
(R.~Larsen, The selected works of T.S.~Spivet, 2009)
\end{flushright}
\end{minipage}
\end{flushright}

\section{Introduction}

Let $(W, I)$ be a \textbf{right-angled} Coxeter system, i.e. a Coxeter system such that $m_{i, j} = 2$ or $m_{i, j} = \infty$ for all $i \neq j$. 
We assume that the generating set $I$ is finite.

Haglund--Paulin have shown that for any tuple of (not necessarily finite) cardinalities $(q_i)_{i \in I}$, there exists a right-angled building of type $(W, I)$ with \textbf{prescribed thicknesses} $(q_i)_{i \in I}$, in the sense that for each $i \in I$, all $i$-panels  have thickness of the same cardinality $q_i$. We refer to \cite[Th.~5.1]{Davis} for a group-theoretic construction of that building. Moreover, such a building is unique up to isomorphism (see Proposition~1.2 in \cite{HP}). A right-angled building satisfying that condition on the panels is called \textbf{semi-regular} (this terminology is motivated by the case of trees). It is \textbf{thick} if $q_i >2$ for all $i \in I$. 

The following shows that the automorphism groups of these buildings provide a large family of simple groups. 

\begin{thm}\label{thm:simple}
Let $X$ be a thick semi-regular building of right-angled type $(W, I)$. 
Assume that $(W, I)$ is irreducible non-spherical. 

Then the group $ \Aut(X)^+$ of type-preserving automorphisms of $X$ is abstractly simple, and acts strongly transitively on $X$.
\end{thm}

Recall that \textbf{strong transitivity} means transitivity on pairs $(c, A)$ consisting of a chamber $c$ and an apartment $A$ containing $c$ (we implicitly refer to the complete apartment system). Haglund and Paulin \cite[Prop.~1.2]{HP} have shown that $\Aut(X)^+$ is chamber-transitive; in fact, the main tools in the proof of Theorem~\ref{thm:simple} rely on their work in an essential way.

\medskip
If $W$ is infinite dihedral, then a building $X$ of type $(W, I)$ with prescribed thicknesses $(q_i)_{i \in I}$ is nothing but a semi-regular tree. In that case the simplicity of the type-preserving automorphism group $G = \Aut(X)^+$ is due to Tits \cite{Tits:arbres}. If  $(W, I)$ is a \textbf{right-angled Fuchsian group} (i.e. if $I = \{1, \dots, r\}$ and $m_{i, j} = 2$ if and only if $|i - j| = 1$ or $r-1$), then a building $X$ of type $(W, I)$ is a \textbf{Bourdon building}, and the simplicity statement is due to Haglund--Paulin \cite{HP:simple}.

\medskip
After this work was completed, K.~Tent informed me that she had obtained independently a proof of \emph{bounded} simplicity  in the case right-angled buildings whose panels are of countable thickness; this stronger simplicity statement means that there is a uniform constant $N$ such that the group can be written as a product of $N$ copies of each of its non-trivial conjugacy classes. In the case of trees, bounded simplicity was proved without any restriction on the thickness by J.~Gismatullin~\cite{Gis}. Another related simplicity theorem was also obtained by N.~Lazarovich~\cite{Laz}; it applies to a large family of groups acting on locally finite, finite-dimensional CAT($0$) cube complexes. It is likely that the special case of Theorem~\ref{thm:simple} concerning locally finite right-angled buildings could also be deduced from \cite{Laz}, using the fact that right-angled buildings can be cubulated. 

\medskip
Notice that a building whose type-preserving automorphism group is chamber-transitive, is necessarily semi-regular. The following is thus immediate from Theorem~\ref{thm:simple}. 

\begin{cor}
Let $X$ be an irreducible thick  right-angled  building of non-spherical type.
If  $ \Aut(X)^+$ is chamber-transitive, then it is strongly transitive and  abstractly simple. 
\end{cor}

In the special case when $X$ is \textbf{locally finite}, i.e. when $q_i < \infty$ for all $i \in I$, the group $G$ endowed with the compact-open topology is a second countable totally disconnected locally compact group. It is compactly generated since it acts chamber-transitively on $X$. In particular Theorem~\ref{thm:simple} provides a large family of compactly generated simple locally compact groups. Our next goal is to describe their rich local structure. 

A general study of the local structure of compactly generated, topologically simple, totally disconnected locally compact groups is initiated in \cite{CRW-short} and \cite{CRW}. The main objects of consideration in that study are the \textbf{locally normal subgroups}, namely the compact subgroups whose normaliser is open. The trivial subgroup, as well as the compact open subgroups, are obviously locally normal, considered as trivial. It is important to observe that a compactly generated,  locally compact group can be topologically simple and nevertheless possess non-trivial locally normal subgroups. Basic examples of such groups are provided by the type-preserving automorphism group of semi-regular locally finite tree. It turns out that the group of type-preserving automorphisms of an arbitrary semi-regular locally finite tree always admits non-trivial locally normal subgroups; some of them even split non-trivially as direct products  (see Lemma~\ref{lem:LocNorm} below). The case of trees has however a special additional property: some compact \emph{open} subgroups split as a direct product of infinite closed subgroups; the corresponding factors are \emph{a fortiori} locally normal and non-trivial. It is thus natural to ask for which right-angled buildings that situation occurs, beyond the case of trees. The following provides a complete answer to this question, implying in particular that open subgroups admit  non-trivial product decompositions only under very special circumstances. 

\begin{thm}\label{thm:product}
Let $X$ be a building of right-angled type $(W, I)$ and prescribed thicknesses $(q_i)_{i \in I}$, with $2 < q_i < \infty $ for all $i \in I$. 
Assume that $(W, I)$ is irreducible non-spherical. 

Then the following assertions are equivalent:
\begin{enumerate}[(i)]
\item All open subgroups of $G = \Aut(X)^+$ are indecomposable.  

\item $G$ is one-ended. 

\item $W$ is one-ended. 
\end{enumerate}

\end{thm}

By \textbf{indecomposable}, we mean the non-existence of a non-trivial direct product decomposition. The  \textbf{set of ends} of a compactly generated locally compact group is defined with respect to compact generating sets in the same way as for discrete groups (see \cite{Abels}). Notice that  Theorem~\ref{thm:product}  establishes a relation between the \emph{local}  structure of $G$ (because the existence of an open subgroup splitting non-trivially as a product can be detected in arbitrarily small identity neighbourhoods) and its \emph{asymptotic} properties. 

The condition that $W$ is one-ended can easily be read on the Coxeter diagram (see Theorem~\ref{thm:product:bis} for a precise formulation). H.~Abels \cite{Abels} has shown that a natural analogue of Stallings' theorem holds for non-discrete locally compact groups. This ensures that $G = \Aut(X)^+$ is one-ended if and only if it does not split non-trivially as an amalgamated free product over a compact open subgroup.

It follows from Theorem~\ref{thm:product} that, if $X$ is a Bourdon building, then compact open subgroups of $\Aut(X)^+$ are indecomposable, but they have locally normal subgroups that split non-trivially as products, all of whose factors are themselves locally normal. With Theorem~\ref{thm:product} at hand, one can construct buildings $X$ of arbitrarily large dimension whose automorphism group has that property. 

\subsection*{Acknowledgement}
I thank Colin Reid and George Willis for numerous inspiring conversations; the main motivation for the present work was in fact provided by the common enterprise initiated in \cite{CRW}. I am grateful to the anonymous referee for a very thorough reading of the paper; his/her numerous comments and suggestions helped much in  correcting  its inaccuracies and improving its readability.  

\section{Projections and parallel residues}

Throughout the paper, we will mostly view a building $X$ of type $(W, I)$ as a $W$-metric space; we refer to \cite{AB} for the basic concepts. Occasionally, geometric arguments will require to consider the Davis realisation of $X$, as defined in \cite{Davis} or \cite[Ch.~12]{AB}. This point of view will be implicit when discussing configuration of walls in a given apartment.  In order to avoid any confusion between these two viewpoints, we will avoid to identify a residue $R$ with the chambers adjacent to it; instead the latter set of chambers is denoted by $\Ch(R)$. 

A fundamental feature of buildings is the existence of combinatorial projections between residues. We briefly recall their basic properties, which will be frequently used in the sequel. All the properties which we do not prove in detail are established in \cite[\S 3.19]{TitsLN}. 

Let $X$ be a building of type $(W, I)$. Given a chamber $c \in \Ch(X)$ and a residue $\sigma$ in $X$, the \textbf{projection} of $c$ on $\sigma$ is the unique chamber of $\Ch(\sigma)$ that is closest to $c$. 
It is denoted by $\proj_\sigma(c)$. For any chamber $d \in \Ch(\sigma)$, there is a minimal gallery from $c$ to $d$ passing through $\proj_\sigma(c)$, and such that the subgallery from $\proj_\sigma(c)$ to $d$ is contained in $\Ch(\sigma)$. Moreover, any apartment containing $c$ and meeting $\Ch(\sigma)$ also contains $\proj_\sigma(c)$. An important property of $\proj$ is that it does not increase the numerical distance between chambers: for all $c, c' \in \Ch(X)$, the numerical distance from $\proj_\sigma(c)$ to $\proj_\sigma(c')$ is bounded above by the numerical distance from $c$ to $c'$. 

If $\sigma$ and $\tau$ are two residues, then the set 
$$\{\proj_\sigma(c) \; | \; c \in \Ch(\tau)\}$$
is the chamber-set of a residue contained in $\sigma$. That residue is denoted by $\proj_\sigma(\tau)$. The rank of $\proj_{\sigma}(\tau)$ is bounded above by the ranks of both  $\sigma$ and $\tau$. 

We shall often use the following crucial property of the projection map; we outsource its statement for the ease of reference. 

\begin{lem}\label{lem:NestedProj}
Let $R, S$ be two residues such that $\Ch(R) \subseteq \Ch(S)$. Then for any residue $\sigma$, we have $\proj_R(\sigma) = \proj_R(\proj_S(\sigma))$. 
\end{lem}

\begin{proof}
See \cite[3.19.5]{TitsLN}.
\end{proof}

Two residues $\sigma$ and $\tau$ are called \textbf{parallel} if $\proj_\sigma(\tau) = \sigma$ and $\proj_\tau(\sigma) = \tau$. In that case, the chamber sets of $\sigma$ and $\tau$ are mutually in bijection under the respective projection maps.  Since the projection map between residues does not increase the rank, it follows that two parallel residues have the same rank. A basic example of parallel residues is provided by   two opposite residues in a spherical building. Another one is provided by the following. 

\begin{lem}\label{lem:ProductRes}
Let $J_1, J_2 \subset I$ be two disjoint subsets with $[J_1, J_2] = 1$. Let $c \in \Ch(X)$. Then 
$$
\Ch(\Res_{J_1 \cup J_2}(c)) = \Ch(\Res_{J_1}(c)) \times \Ch(\Res_{J_2}(c)),
$$ 
and for $i \in \{1, 2\}$, the canonical projection map 
$\Ch(\Res_{J_1 \cup J_2}(c)) \to  \Ch(\Res_{J_i}(c))$ coincides  with the restriction of $\proj_{\Res_{J_i}(c)}$  to $\Ch(\Res_{J_1 \cup J_2}(c))$. In particular, any two $J_i$-residues contained in $\Res_{J_1 \cup J_2}(c)$ are parallel. 
\end{lem}
\begin{proof}
See \cite[Th.~3.10]{Ronan}.
\end{proof}

Given a chamber $c$ and a residue $R$  in $X$, we set $\dist(c, R) = \dist(c, \proj_R(c))$. Given another residue $R'$, we set 
$$\dist(R, R') = \min_{c \in \Ch(R)} \dist(c, R') = \min_{c' \in \Ch(R')} \dist(c', R).$$

\begin{lem}\label{lem:ConstantDistance}
Let $\sigma$ and $\tau$ be parallel residues. For all $x \in \Ch(\sigma)$ and $y \in \Ch(\tau)$, we have $\dist(x, \tau) = \dist(y, \sigma) = \dist(\sigma, \tau)$. 
\end{lem}

\begin{proof}
Let $\Sigma$ be an apartment containing $x$ and $y$. By convexity, it also contains $x' = \proj_\tau(x)$  and $y' = \proj_\sigma(y)$. Since $\sigma$ and $\tau$ are parallel, every wall of $\Sigma$ crossing $\sigma \cap \Sigma$ also crosses $\tau \cap \Sigma$ and vice-versa. It follows that the respective stabilisers of $\sigma \cap \Sigma$ and $\tau \cap \Sigma$  in the Weyl group $W$ coincide. In particular the unique element $w \in W$ mapping $x$ to $y'$ preserves both $\sigma$ and $\tau$. Since  $\sigma$ and $\tau$ are parallel, we have $\proj_{\tau}(y') = y$ so that $w$ maps $x'$ to $y$. Hence $\dist(x, \tau) =\dist(x, x') = \dist(y', y) = \dist(\sigma, y)$. The result follows.
\end{proof}

The relation of parallelism plays a special role among panels. The following criterion will be used frequently. 

\begin{lem}\label{lem:ParallPanels}
Let $\sigma$ and $\sigma'$ be panels.

If two chambers of $\sigma'$ have distinct projections on $\sigma$, then $\sigma$ and $\sigma'$ are parallel. 
\end{lem}
\begin{proof}
If two chambers of $\sigma'$ have distinct projections on $\sigma$, then $\proj_\sigma(\sigma')$ is a panel, which is thus the whole of $\sigma$. Therefore, in an apartment intersecting both $\sigma$ and $\sigma'$, we see that those panels lie on a common wall. It follows that $\proj_{\sigma'}(\sigma)$ cannot be reduced to a single chamber. Hence $\proj_{\sigma'}(\sigma) = \sigma'$ and the result follows.
\end{proof}

The following   result shows that two residues  are parallel if and only if they share the same set of walls in every apartment intersecting them both. This useful criterion allows one to detect parallelism of residues by just looking at parallelism among panels. 

\begin{lem}\label{lem:ParallelPanels}
Let $R$ and $R'$ be two residues. Then $R$ and $R'$ are parallel if and only if for all panels $\sigma$ of $R$ and $\sigma'$ of $R'$, the projections $\proj_{R'}(\sigma)$ and $\proj_{R}(\sigma')$ are both panels. 
\end{lem}

\begin{proof}
The `only if' part is clear from the definition. Assume that $R$ and $R'$ are not parallel. Up to swapping the roles of $R$ and $R'$, we may thus assume that $\proj_R(R')$ is a proper residue of $R$. Let then $c$ and $d$ be a pair of adjacent chambers in $R$ so that $c$ is the projection of some chamber of $R'$ and $d$ is not. Then $c' = \proj_{R'}(c)$ is adjacent to $d' = \proj_{R'}(d)$. If the latter two chambers coincide, then the projection on $R'$ of the panel shared by $c$ and $d$ is a chamber and not a panel, and the desired condition holds. Otherwise the panel shared by $c$ and $d$ is parallel to the panel shared by $c'$ and $d'$ by Lemma~\ref{lem:ParallPanels}. This implies that $c = \proj_R(c')$ and $d=\proj_R(d')$, contradicting the fact that $d$ does not belong to the chamber-set of $\proj_R(R')$.
\end{proof}

Another useful fact is the following.

\begin{lem}\label{lem:ParallelProj}
Let $R$ and $R'$ be two residues. 

Then  $\proj_{R'}(R)$ and $\proj_{R}(R')$ are parallel.
\end{lem}

\begin{proof}
Let $\sigma$ be a panel contained in $\proj_{R}(R')$. Then there is a panel $\sigma' $ in $R'$ such that $\sigma = \proj_R(\sigma')$. It follows that $\sigma$ and $\sigma'$ are parallel. Therefore, we have  
$$
\sigma' = \proj_{\sigma'}(\sigma) = \proj_{\sigma'}(\proj_{R'}(\sigma)),
$$
where the second equality follows from Lemma~\ref{lem:NestedProj}. It follows that $\proj_{R'}(\sigma)$ is a panel. Clearly, we have $\proj_{R'}(\sigma) = \proj_{\proj_{R'}(R)}(\sigma)$. This shows that the projection of $\sigma$ to $\proj_{R'}(R)$ is a panel. 

By symmetry, the projection of any panel of $\proj_{R'}(R)$ to $\proj_{R}(R')$ is also a panel. 
By Lemma~\ref{lem:ParallelPanels}, we infer that $\proj_{R'}(R)$ and $\proj_{R}(R')$ are parallel.
\end{proof}

\medskip

We shall see that parallelism of residues has a very special behaviour in right-angled buildings. For instance, we have the following useful criterion. 

\begin{prop}\label{prop:CriterionParallelRAB}
Let $X$ be a right-angled building of type $(W, I)$. 

\begin{enumerate}[(i)]
\item Two parallel residues have the same type. 

\item Given a residue $R$  of type $J$, a residue $R'$ is parallel to $R$ if and only if  $R'$ is of type $J$ and $R$ and $R'$ are both contained in a residue of type $J \cup J^\perp$. 
\end{enumerate}
\end{prop}
 We recall that  $J^\perp$ is the subset of $I$ defined by 
 $$J^\perp = \{ i \in I \; | \; i \not \in J, \  ij = ji \text{ for all } j \in J\}.$$
In the special case where $J$ is a singleton, say $J = \{j\}$, it is customary to make a slight abuse of notation  and write 
$$j = J   \hspace{1cm} \text{and} \hspace{1cm}  j^\perp = J^\perp$$
when referring to the type of a residue; this should not cause any confusion.

\begin{proof}[Proof of Proposition~\ref{prop:CriterionParallelRAB}]
(i) In a right-angled building, any two panels lying on a common wall in some apartment have the same type. That two parallel residues have the same type is thus a consequence of Lemma~\ref{lem:ParallelPanels}.

\medskip \noindent
(ii) Any two residues of type $J$ in a building of type $J \cup J^\perp$ are parallel. This implies that the `if' part holds. 

Assume now that $R$ and $R'$ are parallel. 
Let $c \in \Ch(R)$ and $c' = \proj_{R'}(c)$. We show by induction on $\dist(c, c')$ that the type of every panel crossed by a minimal gallery from $c$ to $c'$ belongs to $J^\perp$. Let $c = d_0, d_1, \dots, d_n = c'$ be such a minimal gallery. 
Let also $i$ be the type of the panel shared by $c= d_0$ and $d_1$, and let $\sigma_i$ denote that $i$-panel. For any $j \in J$, let also $\sigma_j$ be the $j$-panel of $c$.  By Lemma~\ref{lem:ParallelPanels}, the projection $\sigma'_j = \proj_{R'}(\sigma_j)$ is a panel. The panels $\sigma_j$ and $\sigma'_j$ lie therefore on a common wall in any apartment containing them both. If $i$ and $j$ did not commute, then the wall $\mathcal W_i$ containing the  panel $\sigma_i$ in such an apartment would be disjoint    from the wall $\mathcal W_j$ containing $\sigma_j$. This implies $c'$ is separated from $\mathcal W_j$ by the wall $\mathcal W_i$, which prevents the panel $\sigma'_j$ from lying on $\mathcal W_j$. This shows that $ij = ji$. In other words, we have $i \in J^\perp$. 

Let next $R_1$ be the $J$-residue containing $d_1$, and let $S$ be the $(J \cup \{i\})$-residue containing $c$. Thus $R$ and $R_1$ are both contained in $S$. 

We claim that $R_1$ is parallel to $R'$. 
In order to establish the claim, we first notice that $\proj_{R'}(S) = R'$, since $R \subset S$. By Lemma~\ref{lem:ParallelProj}, the residue $R'=\proj_{R'}(S)$ is parallel to $\proj_{S}(R')$. In particular  $\proj_{S}(R')$ is of type $J$ by part (i). Since $i \in J^\perp$,  all $J$-residues in $S$ contain exactly one chamber of $\sigma_i$. Thus $\sigma_i$ is not contained in $\proj_S(R')$, and it follows that all chambers of $R'$ have the same projection on $\sigma_i$; that projection is the unique chamber of $\Ch(\sigma_i) \cap \Ch(\proj_S(R'))$. By construction, we have $\proj_{\sigma_i}(c') = d_1$; we deduce that $d_1$ belongs to $\Ch(\proj_S(R'))$. This proves that $\proj_S(R')$ is the $J$-residue of $d_1$; it coincides therefore with $R_1$. 

Thus we have shown that $R'=\proj_{R'}(S)$ and that $R_1= \proj_{S}(R')$, and those residues are parallel by  Lemma~\ref{lem:ParallelProj}. The claim stands proven.

The claim implies by induction on $n$ that $R_1$ and $R'$ are contained in a common residue of type $J \cup J^\perp$. That residue must also contain $R$, since   $R$ and $R_1$ are  contained in a common residue of type $J \cup \{i\} \subseteq J \cup J^\perp$. This finishes the proof. 
\end{proof}

\begin{cor}\label{cor:ParallelEquivRel}
Let $X$ be a right-angled building. 

Then parallelism of residues is an equivalence relation. 
\end{cor}

\begin{proof}
This follows from Proposition~\ref{prop:CriterionParallelRAB}(ii).
\end{proof}

We emphasize that parallelism of panels is \emph{not an equivalence relation} in general. In fact, we have the following characterization of right-angled buildings.

\begin{prop}\label{prop:ParallelEquivRel}
Let $X$ be a thick building. 

Then parallelism of residues is an equivalence relation if and only if $X$ is right-angled. 
\end{prop}

\begin{proof}
By Corollary~\ref{cor:ParallelEquivRel}, it suffices to show that if $X$ is not right-angled, then parallelism of panels is not an equivalence relation. If $X$ is not right-angled, then it contains a residue $R$ which is an irreducible  generalized polygon. Let $\sigma$ and $\sigma'$ be two distinct panels of the same type in $R$, at minimal distance from one another. It follows that $\sigma$ and $\sigma'$ are not opposite in $R$, and thus not parallel since they do not lie on a common wall in apartments containing $\sigma$ and $\sigma'$. By \cite[3.30]{TitsLN}, there is a panel $\tau$ in  $R$ which is opposite both $\sigma$ and $\sigma'$. Thus $\sigma$ is parallel to $\tau$ and $\tau$ is parallel to $\sigma'$. Parallelism is thus not a transitive relation. 
\end{proof}

\section{Wall-residues and wings}

Let $X$ be a right-angled building of type $(W, I)$. 

\medskip
Since parallelism of residues is an equivalence relation by Corollary~\ref{cor:ParallelEquivRel}, it is natural to ask what the equivalence classes are. The answer is in fact already provided by Proposition~\ref{prop:CriterionParallelRAB}:  the classes of parallel $J$-residues  are the sets of $J$-residues contained in a common residue of type $J \cup J^\perp$. 

Given a residue $R$ of type $J$, we will denote the unique residue of type $J \cup J^\perp$ containing $R$ by $\overline R$. The special case of panels is the most important one. A residue of the form $\overline \sigma$, with $\sigma$ a panel, will be called a 
\textbf{wall-residue}. 
  
In the   case when $(W, I)$ is a right-angled Fuchsian Coxeter group,  wall-residues  are what Marc Bourdon calls \textbf{wall-trees}, see \cite{Bourdon}. The terminology is motivated by the following observation: if the intersection of a wall-residue with an apartment is non-empty, then it is a wall in that apartment. 

\medskip
Our next step is to show how residues determine a partition of the chamber-set of the ambient building into convex pieces. To this end, we    need some additional terminology and notation. 

To any $c \in \Ch(X)$ and $J \subseteq I$, we associate the set
$$
X_J(c) = \{x \in \Ch(X) \; | \; \proj_{\sigma}(x) = c\},
$$
where $\sigma =\Res_J(c)$ is the $J$-residue of the chamber $c$. We call $X_J(c)$ the \textbf{$J$-wing} containing $c$. If $J = \{i\}$ is a singleton, we write $X_i(c)$ and call it the \textbf{$i$-wing} of $c$. A \textbf{wing} is a $J$-wing for some $J \subseteq I$. The following results record some basic properties of wings. 

\begin{lem}\label{lem:BasicWings}
Let $X$ be right-angled building of type $(W, I)$, let $J \subseteq I$ and $c \in \Ch(X)$. Then we have:
\begin{enumerate}[(i)]
\item $X_J(c) = \bigcap_{i \in J} X_i(c).$

\item
$X_J(c) = X_J(c')$   for all $c' \in X_J(c) \cap \Res_{J \cup J^\perp}(c)$. 

\item $\Res_{J^\perp}(c) =  X_J(c) \cap \Res_{J \cup J^\perp}(c) = \Res_{J^\perp}(c')$  for all $c' \in X_J(c) \cap \Res_{J \cup J^\perp}(c)$. 
\end{enumerate}
\end{lem}

\begin{proof}
(i)  The inclusion $\subseteq$ is clear. To check the reverse inclusion, let $x$ be a chamber whose projection onto $R= \Res_J(c)$ is different from $c$. Then there is a minimal gallery from $x$ to $c$ via $x' = \proj_R(x)$. Let $i$ be type of the last panel  crossed by that gallery, and let $\sigma$ be that panel. By construction we have $\proj_\sigma(x) \neq c$. Moreover, since $x' \neq c$, we have $i \in J$. This implies that $x \not \in X_i(c)$, thereby proving (i).

\medskip \noindent (ii) 
Let $x \in \Ch(X)$ and set $y= \proj_{\Res_{J \cup J^\perp}}(x)$. Let also $R = \Res_J(c)$ and $R' = \Res_J(c')$. By Lemma~\ref{lem:NestedProj}, we have $\proj_{R}(x) = \proj_R(y)$ and $\proj_{R'}(x) = \proj_{R'}(y)$. Moreover, since $\proj_R(c') = c$ by hypothesis, we infer from Lemma~\ref{lem:ProductRes} that $\proj_R(y) =c$ if and only if $\proj_{R'}(y) =c'$. This proves that $X_J(c) = X_J(c')$. 

\medskip \noindent (iii) 
Lemma~\ref{lem:ProductRes} also implies that $\Res_{J^\perp}(c) =  X_J(c) \cap \Res_{J \cup J^\perp}(c)$. The desired equality thus follows using part (ii). 
\end{proof}

\begin{prop}\label{prop:WingConvex}
In a right-angled building, wings are convex. 
\end{prop}

\begin{proof}
Let $X$ be right-angled building of type $(W, I)$. Fix $c  \in \Ch(X)$ and $J \subseteq I$.  By Lemma~\ref{lem:BasicWings}(i) it suffices to prove that a wing of the form $X_i(c)$ with $i \in I$ is convex. Let $\sigma$ be the $i$-panel of $c$. 
Let  also $d, d' \in X_i(c)$ and let $d = d_0, d_1, \dots, d_n = d'$ be a minimal gallery joining them. 

Assume that the gallery is not entirely contained in $X_i(c)$. Let $j $ be the minimal index such that $d_{j+1} \not \in X_i(c)$, and let $j'  $ be the maximal index such that $d_{j'-1} \not \in X_i(c)$. Thus $j' > j$. 

By Lemma~\ref{lem:ParallPanels}, the panel  $\sigma_j$ shared by $d_j$ and $d_{j+1}$ is parallel to $\sigma$. Similarly, so is the panel $\sigma_{j'}$ shared by $d_{j'}$ and $d_{j'-1}$. Therefore, by Proposition~\ref{prop:CriterionParallelRAB}, the set $\Ch(\sigma) \cup \Ch(\sigma_j) \cup \Ch(\sigma_{j'})$ is contained in $\Ch(\overline \sigma)$ (where as above $\overline \sigma$ denotes the $(i \cup i^\perp)$-residue containing $\sigma$). 

For each $k$ between $j$ and $j'$, we now set $d'_k = \proj_{\Res_{i^\perp}(c)}(d_k)$. Notice that by  Lemma~\ref{lem:BasicWings}(iii), we have $\Res_{i^\perp}(c) = \Res_{i^\perp}(d_j) = \Res_{i^\perp}(d_{j'})$. We infer that $d'_{j+1} = d'_j$ and $d'_{j'-1} = d'_{j'}$. Therefore   the sequence 
$$d_j = d'_j = d'_{j+1}, d'_{j+2}, \dots, d'_{j'-2}, d'_{j'-1} = d'_{j'} = d_{j'}$$
is a gallery that is strictly shorter than the given minimal gallery $d_j, d_{j+1}, \dots, d_{j'}$. This  is absurd. 
\end{proof}

By definition of the projection, the set $\Ch(X)$ is the disjoint union of the wings $X_J(d)$ over all $d \in \Ch(\Res_J(c))$. It thus follows from Proposition~\ref{prop:WingConvex} that any residue containing $q$ chambers yields a partition of the building into $q$ convex subsets. 

\medskip
For the sake of future references, we record the following fact. 

\begin{lem}\label{lem:concat}
Let $i \in I$, let $c \in \Ch(X)$ and let $\sigma= \Res_i(c)$.

For any $x \in X_i(c)$ and $x' \not \in X_i(c)$, the gallery from $x$ to $x'$ obtained by concatenating a minimal gallery from $x$ to $\proj_{\overline \sigma}(x)$, a minimal gallery from $\proj_{\overline \sigma}(x)$ to $\proj_{\overline \sigma}(x')$, and a minimal gallery from $\proj_{\overline \sigma}(x')$ to $x'$, is minimal. 
\end{lem}

\begin{proof}
A gallery is minimal if and only if its length equals the numerical distance between its extremities.  Therefore, it suffices to show that there is some minimal gallery from $x$ to $x'$ passing through  $\proj_{\overline \sigma}(x)$ and $\proj_{\overline \sigma}(x')$.

Let $\gamma= (x = x_0, x_1, \dots, x_n = x')$ be a minimal gallery from $x$ to $x'$. Since $x'  \not \in X_i(c)$, the gallery $\gamma$ must cross some panel which is parallel to $\sigma$. By Proposition~\ref{prop:CriterionParallelRAB}, this implies that the gallery $\gamma$ meets the residue $\overline \sigma$. 

Let $j$ (resp. $j'$) be the minimal (resp. maximal) index $k$ such that the chamber $x_k$ of $\gamma$ belongs to $\Ch(\overline \sigma)$. Then there is a minimal gallery $\gamma_j$ from $x$ to $x_j$ (resp. $\gamma_{j'}$ from $x_{j'}$ to $x'$) passing through $\proj_{\overline \sigma}(x)$ (resp. $\proj_{\overline \sigma}(x')$). By concatenating $\gamma_j$ and $\gamma_{j'}$ with the gallery $x_j, x_{j+1}, \dots, x_{j'}$, we obtain a gallery $\tilde \gamma$, of the same length as $\gamma$, and joining $x$ to $x'$. Thus $\tilde \gamma$ is minimal. By construction, it passes through $\proj_{\overline \sigma}(x)$ and $\proj_{\overline \sigma}(x')$.
\end{proof}

Remark that if $\Sigma$ is an apartment of $X$ containing a chamber $c$, then the intersection $X_i(c) \cap \Sigma$ is a half-apartment. The set of half-apartments is partially ordered by inclusion; the following result shows that this order relation is reflected by the ordering of the wings in the ambient building. This will play a crucial role in the subsequent discussions. 

\begin{lem}\label{lem:inclusion}
Let $i, i'\in I$ and  $c, c' \in \Ch(X)$. Suppose that at least one of the following conditions holds. 
\begin{enumerate}[(a)]
\item  $c \in X_{i'}(c')$ and $c' \not \in X_{i}(c)$; moreover $i = i'$ or  $m_{i, i'} = \infty$. 

\item $X_i(c)  \cap \Sigma \subseteq X_{i'}(c')  \cap \Sigma$, where $\Sigma$ is an apartment containing $c$ and $c'$.

\end{enumerate}
Then  $X_i(c) \subseteq X_{i'}(c')$. 
\end{lem}

\begin{proof}
Assume first that (a) holds and let $\Sigma$ be an apartment containing $c$ and $c'$. Let $\mathcal W$  (resp. $\mathcal W'$) be the wall of $\Sigma$ which bounds the half-apartment $X_i(c)  \cap \Sigma$ (resp. $X_{i'}(c')  \cap \Sigma$). The fact that   $i = i'$ or  $m_{i, i'} = \infty$ ensures that the walls $\mathcal W$ and $\mathcal W'$ have trivial intersection (the case $\mathcal W = \mathcal W'$ is excluded in view of Lemma~\ref{lem:BasicWings}(ii)). Therefore the wall $\mathcal W$ is contained in the half-apartment $X_{i'}(c') \cap \Sigma$ because   $c \in X_{i'}(c') \cap \Sigma$. It follows that either $X_i(c)  \cap \Sigma$ or the complementary half-apartment is contained in $X_{i'}(c') \cap \Sigma$. The latter case is excluded, since it would imply that $c' \in X_i(c)  \cap \Sigma$. This proves that (b) holds. 
Hence it suffices to prove the lemma under the hypothesis (b). 

We may assume that $c' \not \in X_i(c)$, since otherwise $X_i(c) \cap \Sigma = X_{i'}(c') \cap \Sigma$, and hence $X_i(c) = X_{i'}(c')$ by Lemma~\ref{lem:BasicWings}(ii). 

Let $\sigma$ (resp. $\sigma'$) be the $i$-panel (resp. $i'$-panel) of $c$ (resp. $c'$). Let $d \in \Ch(X)$ with  $\proj_\sigma(d) = c$. We need to show that  $\proj_{\sigma'}(d) = c'$. Equivalently, for each chamber $c'' \in \Ch(\sigma')$  be different from $c'$, we need to show  that  $\dist(d, c'') = \dist(d, c')+1$. 
Let $\overline \sigma = \Res_{i \cup i^\perp}(c)$.  Let $x = \proj_{\overline \sigma}(c')$ and $y = \proj_{\overline \sigma}(d)$.
By hypothesis (b), and since apartments are convex, both chambers belonging to  $\Sigma \cap \Ch(\sigma')$ have the same projection on $\overline \sigma$, namely $x$. Therefore $\proj_{\overline \sigma}(\sigma') = x$ and, in particular, $\proj_{\overline \sigma}(c'') = x$. 

By assumption, we have $d \in X_i(c)$ and $c' \not \in X_i(c)$. Therefore, Lemma~\ref{lem:concat} implies that  
$$\dist(d, c') = \dist(d, y) + \dist(y, x) + \dist(x, c').$$ 
Moreover, by the claim, we also have 
$$\dist(d, c'') = \dist(d, y) + \dist(y, x) + \dist(x, c'').$$ 
So it suffices to show that $\dist(x, c'') = \dist(x, c')+1$. But Lemma~\ref{lem:concat} applied to $c$ and $c''$ also implies that  
$$\begin{array}{rcl}
\dist(c, x) + \dist(x, c'') & = & \dist(c, c'') \\
& = & \dist(c, c') +1 \\
& = & \dist(c, x) + \dist(x, c') +1,
\end{array}$$ 
whence  $\dist(x, c'') = \dist(x, c') +1$, as   desired. 
\end{proof}

We also need to analyse when a ball or a residue is contained in a given wing. This is the purpose of the next result, whose statements require the following notation. 
We denote by $B(R, n)$   the ball of radius $n$ around $\Ch(R)$, i.e. the collection of all chambers $c$  such that $\dist(c, R) \leq n$. 

\begin{lem}\label{lem:balls:bis}
Let $R$ be a residue, let $i \in I$ and let  $\overline \sigma$ be a  residue of type $i \cup i^\perp$.  Let $R' = \proj_{\overline \sigma}(R)$, let $c \in \Ch(R')$ and $n = \dist(c, R)$. Assume that $\Ch(R') \subseteq X_i(c)$. Then:
\begin{enumerate}[(i)]
\item $B(R, n) \subseteq X_i(c)$.

\item $B(R, n+1) \subseteq X_i(c) \cup \big(\bigcup_{z \in \Ch(R')} \Ch(\Res_i(z))\big)$.
\end{enumerate}

\end{lem}

\begin{proof}
We first claim that $\Ch(R) \subseteq X_i(c)$. Notice that $\Ch(R)$ contains at least one chamber in $X_i(c)$, namely a chamber $x \in \Ch(R)$ such that $\proj_{\overline \sigma}(x) = c$. Therefore, if  $\Ch(R) \not \subseteq X_i(c)$, then  $R$ would contain a panel $\tau$ parallel to the $i$-panel of $c$ by Lemma~\ref{lem:ParallPanels}. Therefore $R' = \proj_{\overline \sigma}(R)$ would contained $\proj_{\overline \sigma}(\tau)$, which is also parallel to the $i$-panel of $c$ by Lemma~\ref{lem:NestedProj}. Notice that $\proj_{\overline \sigma}(\tau)$ is an $i$-panel by Proposition~\ref{prop:CriterionParallelRAB}(i).  Therefore $\Ch(R')$ is not contained in $\Res_{i^\perp}(c)$. By Lemma~\ref{lem:BasicWings}(iii), this implies that $\Ch(R') \not \subseteq X_i(c)$, contradicting the hypothesis. The claim stands proven.  

Choose a chamber $y \in B(R, n+1) - X_i(c)$. 
Let $x = \proj_R(y)$ and let $x= x_0, x_1, \dots, x_{m} =y$ be a minimal gallery. Hence $m = \dist(x, y) = \dist(R, y) \leq n+1$. 
By the claim above, we have $x \in X_i(c)$. On the other hand $y \not \in X_i(c)$ by assumption, so that it makes sense to define  $k_0 = \min \{ \ell \; | \; x_\ell \not \in X_i(c)\}$.
Thus $k_0 >0$ and $x_s \in X_i(c)$ for all $s \in \{0, \dots, k_0-1\}$.

We next observe that the panel  $\sigma'$ shared by $x_{k_0-1}$ and $x_{k_0}$ is parallel to $\sigma = \Res_i(c)$ by Lemma~\ref{lem:ParallPanels}, and is thus of type $i$ by Proposition~\ref{prop:CriterionParallelRAB}(i). Moreover  $x_{k_0-1}$ and $x_{k_0}$ both belong to $ \Ch(\overline \sigma)$ by Proposition~\ref{prop:CriterionParallelRAB}(ii).  
 In particular we have
$$n \geq m-1 \geq k_0-1 = \dist(x, x_{k_0-1})  \geq \dist(x, \proj_{\overline \sigma}(x)).$$
There is a minimal gallery from $x$ to $\proj_{\overline \sigma}(x)$ passing through $x' = \proj_R(\proj_{\overline \sigma}(x)))$. The residues  $\proj_R(\overline \sigma)$ and $R'$ are parallel by Lemma~\ref{lem:ParallelProj}.  
Therefore, we deduce from Lemma~\ref{lem:ConstantDistance} that  
$$\dist(x', \proj_{\overline \sigma}(x)) =  \dist(x', \proj_{\overline \sigma}(x'))=\dist( \proj_R(\overline \sigma), R') = \dist( R, c)=n.$$ 
This implies that $\dist(x, \proj_{\overline \sigma}(x)) \geq n$. From the sequence of inequalities above, we deduce  that  $m = k_0 = n+1$. Part (i) follows.  

Moreover, since $n = \dist(x, x_{k_0-1} ) \geq  \dist(x, \proj_{\overline \sigma}(x)) \geq n$, we have $x_{k_0 -1} = \proj_{\overline \sigma}(x)$ and hence $x_{k_0 - 1} \in R'$.  Thus $y \in \Ch(\sigma') \subseteq 
\big(\bigcup_{z \in \Ch(R')} \Ch(\Res_i(z))\big)$. This proves (ii).
\end{proof}

\begin{cor}\label{cor:balls}
Let $i \in I$, let $c,x \in \Ch(X)$ and  $n = \dist(c,x)$. Let also $\sigma = \Res_i(c)$ and $\overline \sigma = \Res_{i \cup i^\perp}(c)$. 
If $\proj_{\overline \sigma}(x) = c$, then   $B(x, n+1) \subseteq  X_i(c) \cup \Ch(\sigma) $. 
\end{cor}

\begin{proof}
Let $R = \{x\}$. Then $ \proj_{\overline \sigma}(R) = \{c\} \subseteq X_i(c)$. Thus the desired inclusion follows from Lemma~\ref{lem:balls:bis}.
\end{proof}

\begin{cor}\label{cor:ResidueWing}
Let $J \subseteq I$ and $i \in I - J$. Given a $J$-residue $R$ and a chamber $c \in \Ch(R)$, we have $\Ch(R) \subseteq X_i(c)$. 
\end{cor}

\begin{proof}
Let $\overline \sigma = \Res_{i \cup i^\perp}(c)$ and $R' = \proj_{\overline \sigma}(R)$. Since $c \in \Ch(R) \cap \Ch(\overline \sigma)$, we have $R' = R \cap \overline \sigma$. 
Recall from Lemma~\ref{lem:BasicWings}(iii) that $X_i(c) \cap \Ch(\overline \sigma) = \Res_{i^\perp}(c)$. Therefore, if $\Ch(R')$ were not contained in $X_i(c)$, it would contain an $i$-panel. However the type of $R'$ is a subset of $J$, and therefore does not contain $i$ by hypothesis. This shows that $\Ch(R') \subseteq X_i(c)$. Applying Lemma~\ref{lem:balls:bis}(i) with $n=0$, we obtain $\Ch(R) \subseteq X_i(c)$, as required.
\end{proof}

\section{Extending local automorphisms}

The following important result was shown by Haglund--Paulin. 

\begin{prop}[Haglund--Paulin]\label{prop:extension}
Let $X$ be a semi-regular right-angled building. For any residue $R$ of $X$ and any $\alpha \in \Aut(R)^+$, there is $\tilde \alpha \in \Aut(X)^+$ stabilising $R$ and such that  $\tilde \alpha |_{\Ch(R)} = \alpha$.
\end{prop}
\begin{proof}
See Proposition~5.1 in \cite{HP}.
\end{proof}

In other words, this means that the canonical homomorphism $\Stab_{\Aut(X)^+}(R) \to \Aut(R)^+$ is surjective. 

It will be important for our purposes to ensure that  the extension constructed in Proposition~\ref{prop:extension} can be chosen to satisfy some additional constraints. In particular, we record the following. 

\begin{prop}\label{prop:Panel:extension}
Let $X$ be a semi-regular right-angled building of type $(W, I)$.
Let $i \in I$  and $\sigma$ be an $i$-panel. 

Given any permutation $\alpha \in \Sym(\Ch(\sigma))$, there is $\tilde \alpha \in \Aut(X)^+$ stabilising $\sigma$ satisfying the following two conditions:

\begin{enumerate}[(i)]

\item $\tilde \alpha |_{\Ch(\sigma)} = \alpha$;

\item $\tilde \alpha$ fixes all chambers of $X$ whose projection to $\sigma$ is fixed by $\alpha$.
\end{enumerate}

\end{prop}

\begin{proof}
Let $c_0 \in \Ch(\sigma)$ and $\sigma^\perp = \Res_{i^\perp}(c_0)$. Then $\Ch(\overline \sigma) = \Ch(\sigma) \times \Ch(\sigma^\perp)$ by Lemma~\ref{lem:ProductRes}. 
We define $\beta \in \Aut(\overline \sigma)^+$ as $\beta = \alpha \times \mathrm{Id}$. By Proposition~\ref{prop:extension}, the automorphism $\beta$ of $\overline \sigma$ extends to some (type-preserving) automorphism $\tilde \beta$ of $X$.

We now define a map $\tilde \alpha \colon \Ch(X) \to \Ch(X)$  as follows: for each $c \in \Ch(X)$, we set
$$ 
\tilde \alpha(c) = \left\{
\begin{array}{ll}
c & \text{if } \alpha(\proj_\sigma(c))=\proj_\sigma(c);\\
\tilde \beta(c) & \text{otherwise}.
\end{array}
\right.
$$
Clearly the map $\tilde \alpha$ satisfies the desired condition (ii). Moreover, we have $\tilde \alpha  |_{\Ch(\overline \sigma)} = \beta$, from which it follows that condition (i) holds as well. 

It remains to check that $\tilde \alpha$ is an automorphism. 
To this end, let $x$ and $y$ be any two chambers and denote by $x'$ and $y'$ their projections on $\sigma$. 

If $x' = y'$, then we have either $(\tilde \alpha (x), \tilde \alpha(y)) = (x, y)$, or $(\tilde \alpha (x), \tilde \alpha(y)) =(\tilde \beta (x), \tilde \beta(y))$. In both cases, it follows that $\tilde \alpha$ preserves the Weyl-distance from $x$ to $y$. 

Assume now that $x' \neq y'$. Let then $x''$ and $y''$ denote the projections of $x$ and $y$ on $\overline \sigma$. By Lemma~\ref{lem:concat}, it suffices to show that $\tilde \alpha$ preserves the Weyl-distance from $x$ to $x''$, the Weyl-distance from $x''$ to $y''$ and the Weyl-distance from $y''$ to $y$. Since wings are convex by Proposition~\ref{prop:WingConvex}, and since the restriction of $\tilde \alpha$ on each wing of $\sigma$ preserves the Weyl-distance, it follows that  $\tilde \alpha$ preserves the Weyl-distance from $x$ to $x''$ and from $y''$ to $y$. That the Weyl-distance from $x''$ to $y''$ is preserved is clear since the restriction of $\tilde \alpha$ to $\Ch(\overline \sigma)$ is the automorphism $\beta$. 

This proves that $\tilde \alpha$ preserves the Weyl-distance from $x$ to $y$. Thus $\tilde \alpha$ is an automorphism. 
\end{proof}

\section{Fixators of wings}

As before, let $X$ be a right-angled building of type $(W, I)$. 

\medskip
The subsets $X_i(c)$ are analogues of half-trees in the case $W$ is infinite dihedral. In view of this analogy, we shall consider the subgroups of $\Aut(X)^+$ denoted by $V_i(c)$ and $U_i(c)$, consisting respectively  of automorphisms supported on $X_i(c)$ and on its complement. In symbols, this yields 
$$
U_i(c) = \{g \in \Aut(X)^+ \; | \; g(x) = x \ \text{for all } x \in X_i(c)\},
$$
and 
$$
V_i(c) = \{g \in \Aut(X)^+ \; | \; g(x) = x \ \text{for all } x \not \in  X_i(c)\}.
$$
Clearly $U_i(c)$ and  $V_i(c)$ both fix $c$ and stabilise $\sigma$.  Moreover they commute and have trivial intersection, since their supports are disjoint.  The following implies that they are both non-trivial. 

\begin{lem}\label{lem:NonAb}
Assume that $X$ is thick and semi-regular. Let $i \in I$ be such that $i \cup i^\perp \neq I$. Then for all $c \in \Ch(X)$, the groups $U_i(c)$ and $V_i(c)$ are non-abelian. 
\end{lem}

\begin{proof}
By hypothesis, there exists $j \in I$ not contained in $i \cup i^\perp$. Let $x \neq c$ be a chamber $j$-adjacent to $c$. Then $X_j(x) \subset X_i(c)$ by Lemma~\ref{lem:inclusion}. This implies that $U_i(c)$ fixes pointwise $X_j(x)$ for all chambers $x \neq c$ that are $j$-adjacent to but different from $c$. In particular $U_i(c) $ is contained in $V_j(c)$. Likewise, since $i \not \in j \cup j^\perp$, we have $U_j(c) \leq V_i(c)$. 
In view of the symmetry between $i$ and $j$, it only remains to show that $U_i(c)$ is not abelian. 

Proposition~\ref{prop:Panel:extension} implies that $U_j(c)$   is non-trivial; so is thus  $V_i(x)$  for all $x \in \Ch(X)$ in view of what we have just observed. 

For each $c' \neq c$   that is $i$-adjacent to $c$, the group $V_i(c')$ is contained in $U_i(c)$. Moreover, if $c' , c''$ are two distinct such chambers, the groups $V_i(c')$ and $V_i(c'')$ are different since they are non-trivial and have disjoint supports. By Proposition~\ref{prop:Panel:extension}, there is $u \in U_i(c)$ mapping $c'$ to $c''$. We then have 
$uV_i(c') u\inv = V_i(c'') \neq V_i(c').$ In particular $u$ does not commute with $V_i(c')$, which proves that $U_i(c)$ is not abelian. 
\end{proof}

Given $G \leq \Aut(X)$, the pointwise stabiliser of the chamber set $\Ch(R)$ of a residue $R$ is denoted by $\Fix_G(R)$. We shall next describe how the groups $U_i(c)$ and $V_i(c)$  provide  convenient generating sets for the pointwise stabilisers of residues  in $X$. We start with wall-residues; the case of spherical residues is postponed to Proposition~\ref{prop:Fix} below.

\begin{prop}\label{prop:ProductDec}
Let $X$ be a right-angled building of type $(W, I)$. 
Let $c \in \Ch(X)$ and $i \in I$, and let $R = \Res_{i \cup i^\perp}(c)$ be the residue of type $i \cup i^\perp$ of $c$. 

Then we have 
$$
\Fix_{\Aut(X)^+}(R) = \prod_{d \sim_i c} V_i(d).
$$
\end{prop}

We will use the following subsidiary fact. 
\begin{lem}\label{lem:InfProd}
Let $n >0$ be an integer, let $C, W$ be sets and let $\delta \colon C^n \to W$ be a map. Let $G$ denote the group of all permutations $g \in \mathrm{Sym}(C)$ such that $\delta(g(x_1), \dots, g(x_n)) = \delta(x_1, \dots, x_n)$ for all $(x_1, \dots , x_n) \in C^n$. Let moreover $(V_i)_{i \in I}$ a collection of groups  indexed by a set $I$, and for all $i \in I$, let $\varphi_i \colon V_i \to G$ be an injective homomorphism  such that for all $i \neq j$, the subgroups $\varphi_i(V_i)$ and $\varphi_j(V_j)$ have disjoint supports. Then there is a unique injective homomorphism  
$$ \varphi \colon \prod_{j \in I} V_j \to G$$
such that for all $i \in I$, the composed map $\varphi \circ \iota_i  = \varphi_i$, where $\iota_i \colon V_i \to  \prod_{j \in I} V_j$ is the canonical inclusion. 
\end{lem}

The only relevant case for this paper is when $C$ is the chamber set of a building $X$ and $\delta \colon C \times C \to W$ is the Weyl-distance. In that case, the group $G$ from Lemma~\ref{lem:InfProd} coincides with the group $\Aut(X)^+$ of type-preserving automorphisms of $X$. 

\begin{proof}[Proof of Lemma~\ref{lem:InfProd}]
The uniqueness of $\varphi$ is clear; we focus on the existence proof. 
Set $V = \prod_{j \in I} V_j$ and let $g = (g_j)_{j \in I} \in V$. Given $x \in C$, there is at most one index $j \in I$ such that $\varphi_j(V_j)$ does not fix $x$, since the subgroups   $\varphi_i(V_i)$ have disjoint supports. We set $\varphi(g)(x) = \varphi_j(g_j)(x)$ if there exists such a $j \in I$, and  $\varphi(g)(x) = x$ otherwise. This defines a homomorphism $\varphi \colon \prod_{i \in I} V_i \to \mathrm{Sym}(C)$ such that $\varphi \circ \iota_i  = \varphi_i$ for all $i \in I$. It is injective, since $\varphi(g) = 1$ implies that $\varphi_i(g_i) = 1$ for all $i$, and hence $g_i = 1$ for all $i$ since all $\varphi_i$ are injective by hypohtesis. It remains to prove that $\varphi(g) \in G$. Given  $(x_1, \dots, x_n) \in C^n$, let $J \subseteq I$ be the (possibly empty) subset consisting all the indices $j \in I$ such that $\varphi_j(g_j)$ does not fix all elements of $\{x_1, \dots, x_n\}$. Thus $J$ is finite of cardinality~$\leq n$. Let $g_J$ denote the product of  the elements $\varphi_j(g_j) \in G$ over all $j \in J$, in an arbitrary order; if $J = \varnothing$, we set $g_J = 1$. Since two distinct  subgroups $\varphi_i(V_i)$ and $\varphi_j(V_j)$ have disjoint supports, they commute, and it follows that the product $g_J$ is independent of the chosen order. Moreover,  we have  $\varphi(g)(x_i) = g_J(x_i)$ for all $i \in \{1, \dots, n\}$. Since $g_J \in G$, we infer that $\delta(\varphi(g)(x_1), \dots, \varphi(g)(x_n)) = \delta(g_J(x_1), \dots, g_J(x_n)) = \delta(x_1, \dots, x_n)$, as desired.
\end{proof}

\begin{proof}[Proof of Proposition~\ref{prop:ProductDec}]
Let $d \sim_i c$. Given any $x \in \Ch(R)$, we deduce from Lemma~\ref{lem:ProductRes} that $V_i(d)$ fixes all chambers of the $i$-panel of $x$ different from the projection of $c$. Hence $V_i(d)$ fixes all chambers of that panel. This proves that $V_i(d)$ is contained in  $\Fix_{\Aut(X)^+}(R)$. 

Remark that for two different chambers $d, d' $ that are $i$-adjacent to $c$, the groups $V_i(d)$ and $V_i(d')$  have disjoint supports. From Lemma~\ref{lem:InfProd}, we deduce that 
the (possibly infinite) direct product  $\prod_{d \sim_i c} V_i(d)$ is contained in $\Fix_{\Aut(X)^+}(R)$. 

It remains to show that every element $g \in \Fix_{\Aut(X)^+}(R)$ belongs to the product  $\prod_{d \sim_i c} V_i(d)$. In order to see this, fix $g \in \Fix_{\Aut(X)^+}(R)$ and $d \sim_i c$, and consider the permutation $g_d$ of $\Ch(X)$ defined by
$$g_d \colon  \Ch(X) \to \Ch(X)  \colon x \mapsto 
\left\{ 
\begin{array}{ll}
g(x) & \text{if } x \in X_i(d) \\
x & \text{otherwise}.
\end{array}
\right.
$$
We claim that $g_d \in V_i(d)$. To see this, let $x, y \in \Ch(X)$ and let $\delta \colon \Ch(X) \times \Ch(X) \to W$ denote the Weyl-distance. We need to show that $\delta(g_d(x), g_d(y)) = \delta(x, y)$. By the definition of $g_d$, it suffices to consider the case when  $x \in X_i(d)$ and $y \not \in X_i(d)$ (or vice-versa). By Lemma~\ref{lem:concat}, we have 
$$\delta(x, y ) = \delta(x, x') \delta (x', y') \delta(y', y),$$
where $x' = \proj_R(x)$, $y' = \proj_R(y)$ and $R = \Res_{i \cup i^\perp}(d)$. Moreover, the element $g \in \Aut(X)^+$ fixes $ x', y$ and $y'$ and preserves $R$ and $X_i(d)$. Thus we have $\proj_R(g_d(x)) = \proj_R(g(x))= x'$ and,  invoking Lemma~\ref{lem:concat} once more, we deduce 
$$
\begin{array}{rcl}
\delta(g_d(x),  g_d(y))  &= &  \delta( g(x), y)\\
& = & \delta(g(x), x') \delta(x', y') \delta(y', y) \\
& = & \delta(x, x') \delta(x', y') \delta(y', y) \\
& = &  \delta(x, y)
\end{array}
$$
as desired. Thus $g_d$ is a type-preserving automorphism of $X$. By construction, we have $g_d \in V_i(d)$. Moreover the tuple  $(g_d)_{d \sim_i c}$, which is an element of the direct product $ \prod_{d \sim_i c} V_i(d)$, coincides with $g$. Therefore $g \in  \prod_{d \sim_i c} V_i(d)$. 
\end{proof}

\section{Strong transitivity}

\begin{prop}\label{prop:StrongTrans}
Let $X$ be a semi-regular right-angled building. 

Then the group $\Aut(X)^+$ is strongly transitive on $X$. 
\end{prop}

We need the following basic consequence of Proposition~\ref{prop:Panel:extension}.

\begin{lem}\label{lem:U:residue}
Let $X$ be a semi-regular right-angled building of type $(W, I)$. 
Let $R$ be a residue and let $n, t$ be non-negative integers. For all $s \in \{ 1, \dots, t\}$, let also:
\begin{itemize}
\item $i_s \in I$, 

\item $\overline{ \sigma_s}$ be a residue of type $i_s \cup i_s^\perp$ such that $\dist(R, \overline{ \sigma_s}) = n$,

\item $c_s \in \Ch(R'_s)$, where $R'_s  = \proj_{\overline{\sigma_s}}(R)$. 

\item $\pi_s$ be a  permutation of $\Ch(\sigma_s)$ fixing $c_s$, where  $\sigma_s = \Res_{i_s}(c_s)$. 
\end{itemize}
Assume that the pairs  
$(\overline{ \sigma_1}, i_1), \dots , (\overline{ \sigma_t}, i_t)$  are pairwise distinct, and that $\Ch(R'_s) \subseteq X_{i_s}(c_s)$ for all $s \in \{1, \dots, t\}$. 

 Then  there is $g \in \la U_{i_s}(c_s) \; | \; s = 1, \dots, t\ra$ such that $g|_{\Ch(\sigma_s)} =  \pi_s$ for all $s$. Moreover $g$ fixes pointwise the set $B(R, n+1) - \bigcup_{s=1}^t  \bigcup_{z \in \Ch(R'_s)} \Ch(\Res_{i_s}(z))$.
\end{lem}

\begin{proof}
Let $s \in \{ 1, \dots, t\}$. By Proposition~\ref{prop:Panel:extension}, there exists  $g_s \in U_{i_s}(c_s)$ with $g_s |_{\Ch(\sigma_s)} = \pi_s$. 
By Lemma~\ref{lem:balls:bis}, every element of $U_{i_s}(c_s)$ fixes pointwise the set $B(R, n+1) - \big(\bigcup_{z \in \Ch(R'_s)} \Ch(\Res_i(z))\big)$. 

Let now $s' \neq s$. If $\sigma_{s'}$ were parallel to $\sigma_s$, we would have $i_s = i_{s'}$ and $ \overline{ \sigma_s} =  \overline{ \sigma_{s'}}$ by Proposition~\ref{prop:CriterionParallelRAB}, contradicting our hypotheses. Therefore $\proj_{\sigma_{s}}(\sigma_{s'})$ is a single chamber (see  Lemma~\ref{lem:ParallPanels}). Moreover we have 
$c_{s'} \in \Ch(\sigma_{s'}) \cap B(R, n)$, and  $B(R, n) \subseteq X_{i_s}(c_s)$ by Lemma~\ref{lem:balls:bis}. We infer that $\proj_{\sigma_{s}}(\sigma_{s'}) = c_s$ or, equivalently, that $\Ch(\sigma_{s'}) \subseteq X_{i_s}(c_s)$. Therefore $\Ch(\sigma_{s'})$ is pointwise fixed by $g_s$. It follows that the product $g = g_1 \dots g_t$ enjoys the desired properties. 
\end{proof}

In order to facilitate future references, we state the following special case separately. 

\begin{lem}\label{lem:U}
Let $X$ be a semi-regular right-angled building of type $(W, I)$. 
Let $x \in \Ch(X)$ and let $n, t$ be non-negative integers. For all $s \in \{ 1, \dots, t\}$, let also:
\begin{itemize}
\item $c_s  \in \Ch(X)$ be such that $\dist(x, c_s) = n$, 

\item $i_s \in I$ be such that  $\proj_{\overline{ \sigma_s}}(x) = c_s$, where $\overline{\sigma_s} = \Res_{i_s \cup i_s^\perp}(c_s)$,

\item $\pi_s$ be a  permutation of $\Ch(\sigma_s)$ fixing $c_s$, where  $\sigma_s = \Res_{i_s}(c_s)$. 
\end{itemize}
Assume that the pairs $(c_1, i_1), \dots, (c_t, i_t)$ are pairwise distinct.   Then  there is $g \in \la U_{i_s}(c_s) \; | \; s = 1, \dots, t\ra$ whose restriction to $\Ch(\sigma_s)$ is $\pi_s$ for all $s$. Moreover $g$ fixes pointwise the set $B(x, n+1) - \bigcup_{s=1}^t \Ch(\sigma_s)$.
\end{lem}

\begin{proof}
Since  $\proj_{\overline{ \sigma_s}}(x) = c_s$, we have $\dist(x, \overline{ \sigma_s}) = \dist(x, c_s) = n$. Clearly $c_s \in X_{i_s}(c_s)$. Moreover, if $(\overline{ \sigma_s}, i_s) = (\overline{ \sigma_{s'}}, i_{s'})$, then $(c_s, i_s) = (c_{s'}, i_{s'})$ and hence $s =s'$ by hypothesis. Thus the desired conclusion follows from Lemma~\ref{lem:U:residue}.
\end{proof}

\begin{proof}[Proof of Proposition~\ref{prop:StrongTrans}]
As observed by Haglund--Paulin \cite{HP}, Proposition~\ref{prop:extension} readily implies that $\Aut(X)^+$ is chamber transitive. We need to show that given a chamber $c \in \Ch(X)$ and two apartments $A, A'$ containing $c$, there is an element $g \in \Aut(X)^+$ fixing $c$ and mapping $A$ to $A'$. 

Set $g_0 = \mathrm{Id}$ and let $n > 0$. We shall construct by induction on $n$ an element $g_n \in \Aut(X)^+$ with the following properties:
\begin{itemize}
\item $g_n$ fixes pointwise the ball of radius $n-1$ around $c$;

\item $g_n g_{n-1} \dots g_0(A) \cap A' \supseteq B(c, n) \cap A'$, where $B(c, n)$ is the ball of radius $n$ around $c$. 
\end{itemize}
The first property ensures that the sequence $(g_n g_{n-1} \dots g_0)_{n \geq 0}$ pointwise converges to a well defined automorphism $g_\infty \in \Aut(X)^+$. The second property yields $g_\infty(A) = A'$, as desired. 

Let $n \geq 0$,  and suppose that $g_0, g_1, \dots, g_n$ have already been constructed. Set $A_n =  g_n g_{n-1} \dots g_0(A)$. Thus $A_n \cap A' $ contains every chamber of $A'$ at distance at most $n$ from $c$. 

We need to construct an automorphism $g_{n+1} \in \Aut(X)^+$ fixing $B(c, n)$ pointwise and such that $g_{n+1}(A_n) \cap A'$ contains $B(c, n+1) \cap A'$. 

Let $E $ be the set of those chambers in $B(c, n+1) \cap A'$ that are not contained in $A_n$. Notice that $E$ is finite (since $B(c, n+1) \cap A'$ is so) and that every chamber in $E$ is at distance $n+1$ from $c$. 

If  $E$ is empty, then we set $g_{n+1} = \mathrm{Id}$ and we are done. Otherwise we enumerate $E = \{x'_1, \dots, x'_t\}$ and consider $s \in \{1, \dots, t\}$.  Let $y_s$  be the first chamber different from $x'_s$ on a minimal gallery from $x'_s$ to $c$. Thus $\dist(c, y_s) = n$ and $y_s \in B(c, n) \cap A'$, hence $y_s \in A_n$. Let  $\sigma_s$ be the panel shared by $x'_s$ and $y_s$ and  let $i_s \in I$ be its type. The pairs $(y_s, i_s)$ are pairwise distinct since $(y_{s_1}, i_{s_1}) = (y_{s_2}, i_{s_2})$ in the apartment $A'$ implies that $x'_{s_1} = x'_{s_2}$ and $s_1 = s_2$. 
Finally,  let $x_s \in A_n$ be the unique chamber which is $i_s$-adjacent to, but different from, $y_s$. 

We claim that $\proj_{\overline{\sigma_s}}(c) = y_s$. In order to establish this, consider $z_s = \proj_{\overline{\sigma_s}}(c)$. If $z_s \neq y_s$, then $\dist(c, z_s) < \dist(c, y_s) = n$. Therefore the unique chamber of $z'_s \in A'$ which $i_s$-adjacent to, but different from, $z_s$, also belong to $A_n$. Since apartments are convex, we infer that $\proj_{\sigma_s}(z'_s) \in A_n \cap A'$. On the other hand we have $\proj_{\sigma_s}(z'_s) = x'_s$ by Lemma~\ref{lem:ProductRes}. This contradicts the fact that $x'_s \not \in A_n$, and the claim stands proven.

We are thus in a position to invoke Lemma~\ref{lem:U}. This yields  an element $g_{n+1} \in \la U_{i_s}(y_s) \; | \; s = 1, \dots, t\ra$ which maps $x_s$ to $x'_s$ for all $s$, and fixes $B(c, n)$ pointwise. It follows that $g_{n+1} $ has the requested properties, and we are done.
\end{proof}

We are thus in a position to invoke Tits' transitivity lemma:

\begin{cor}\label{cor:TitsTrans}
Let $X$ be a thick semi-regular right-angled building of irreducible type.

Then every non-trivial normal subgroup of $\Aut(X)^+$ is transitive on $\Ch(X)$. 
\end{cor}
\begin{proof}
Since $\Aut(X)^+$ is strongly transitive by Proposition~\ref{prop:StrongTrans}, this follows from Proposition~2.5 in \cite{Tits64}.
\end{proof}

In case when $X$ is locally finite, the strong transitivity guaranteed by Proposition~\ref{prop:StrongTrans} is already enough to ensure that the intersection of all non-trivial closed normal subgroups of $\Aut(X)^+$ is non-trivial, topologically simple and cocompact, see \cite[Corollary~3.1]{CM}. This is of course a much weaker conclusion than Theorem~\ref{thm:simple}.

\section{Simplicity of the automorphism group}

The following result is established by a similar argument as in Tits' commutator lemma (Lemma~4.3 in \cite{Tits:arbres} or Lemma~6.2 in \cite{HP:simple}).

\begin{lem}\label{lem:TitsCommutator}
Let $X$ be a right-angled building of type $(W, I)$. Let $\sigma$ be a panel of type $i \in I$, let $c, c' \in \Ch(\sigma)$ be two distinct chambers, and let $g \in \Aut(X)^+$ be such that $g(c)$ is $j$-adjacent to $c'$, for some $j \in I$ with $m_{i, j} = \infty$. 

Then, for each element $h \in  \prod_{d \in \Ch(\sigma) \setminus \{c, c'\}} V_i(d)$, there exists $x \in \Aut(X)^+$ such that 
$h = [x, g] = x g x\inv g\inv$.
\end{lem}

\begin{proof}
Let $V_0 =  \prod_{d \in \Ch(\sigma) \setminus \{c, c'\}} V_i(d)$, and remark that $V_0 \leq G$ by Proposition~\ref{prop:ProductDec}. 
For each $n \geq 0$, we also set $\sigma_n = g^n(\sigma)$, $c_n = g^n(c)$,  $c'_n = g^n(c')$ and $V_n = g^n V_0 g^{-n}$. 

For each $n \geq 0$,  the support of $V_n$ is contained in $\bigcup_{d \in \Ch(\sigma_n) \setminus \{c_n, c'_n\}} X_i(d)$. Given $d \in \Ch(\sigma_n) \setminus \{c_n, c'_n\} $ and $m > n$, we have $d \in X_i(c_m)$ and $c_m \not \in X_i(d)$. Therefore $X_i(d) \subset X_i(c_m)$ by Lemma~\ref{lem:inclusion}. This implies that the sets 
$\bigcup_{d \in \Ch(\sigma_n) \setminus \{c_n, c'_n\}} X_i(d)$ 
and 
$\bigcup_{d \in \Ch(\sigma_m) \setminus \{c_m, c'_m\}} X_i(d)$
are disjoint. In other words, we have shown that  for $m > n \geq 0$, the subgroups $V_m$ and $V_n$ have disjoint support. By Lemma~\ref{lem:InfProd}, it follows that the product $V = \prod_{n \geq 0} V_n$ is a subgroup of $G$. Moreover, we have $gV_ng^{-1} = V_{n+1}$ for all $n \geq 0$. 

Given any $h \in V_0$, we set  $x_n = g^n h g^{-n}$ for all $n \geq 0$. Then the tuple $x = (x_n)_{n \geq 0}$ is an element of $V \leq G$. So is thus the commutator $[x, g]$. Moreover, denoting by $y_n$ the $n^{\mathrm{th}}$ component  of an element $y \in V$  according to the  decomposition $V = \prod_{n \geq 0} V_n$, we have $[x, g]_n = x_n (g x\inv g\inv)_n$ for all $n \geq 0$. Hence  $[x, g]_0=h$ and $[x, g]_n = x_n g x_{n-1}\inv g\inv = x_n x_n\inv = 1$ for all $n>0$. Thus $[x,g] = h$, as required. 
\end{proof}

We record the following consequence of Lemma~\ref{lem:TitsCommutator}, which is a crucial ingredient for the proof of Theorem~\ref{thm:simple}. 

\begin{lem}\label{lem:normal}
Let $X$ be a right-angled building of type $(W, I)$.  Assume that the Coxeter system $(W, I)$ is irreducible and that $X$ is thick. 

Then for any wall-residue $R$,   every non-trivial normal subgroup of $\Aut(X)^+$ contains  
$\Fix_{\Aut(X)^+}(R)$. 
\end{lem}

\begin{proof}
Let $N < \Aut(X)^+$ be a non-trivial normal subgroup. 

Let $\sigma$ be a panel of type $i \in I$ and $R = \overline \sigma$ be the corresponding wall-residue. Choose two distinct chambers $c, c' \in \Ch(\sigma)$.  Since $(W, I)$ is irreducible and non-spherical, there exist $j \in I $ such that $m_{i, j} = \infty$. By Corollary~\ref{cor:TitsTrans}, there is $g \in N$ such that $g(c)$ is $j$-adjacent to $c'$. In view of Lemma~\ref{lem:TitsCommutator}, we deduce that $\prod_{d \in \Ch(\sigma) \setminus \{c, c'\}} V_i(d)$ is contained in 
$N$. 

Since the latter holds for all pairs $\{c, c'\}   \subset \Ch(\sigma)$ and since $X$ is thick, we deduce that $V_i(c)$ and $V_i(c')$ are also contained in $N$. Therefore, so is $\Fix_{\Aut(X)^+}(R)$ by Proposition~\ref{prop:ProductDec}.
\end{proof}

We are now ready to complete the proof of simplicity.

\begin{proof}[Proof of Theorem~\ref{thm:simple}]
Let $N \neq 1$ be a non-trivial normal subgroup of $G = \Aut(X)^+$. By Corollary~\ref{cor:TitsTrans}, the group $N$ is transitive on $\Ch(X)$. Since $G$ is strongly transitive on $X$, it is naturally endowed with a $BN$-pair. Therefore, if we show that $N$ contains the full stabiliser $\Stab_G(R)$ of some residue $R$, then it will follow from \cite[Proposition~2.2]{Tits64} that $N$ itself is the stabiliser of some residue. The transitivity of $N$ on $\Ch(X)$ forces that residue to be the whole building $X$, whence $N = G$ as required. Therefore, the desired conclusion will follow provided we show that $N $ contains $\Stab_G(R)$ of some residue $R$. This is the final of the following series of claims.

\begin{claim}\label{cl:1}
For any residue $R$ of irreducible type, the stabiliser $\Stab_N(R)$ maps surjectively to $\Aut(R)^+$.
\end{claim}

In order to prove the claim, we first observe that given two chambers $c, c' \in \Ch(R)$, any element of $G$ mapping $c$ to $c'$ must stabilise $R$. Since $N$ is chamber-transitive, it follows that for any residue $R$, the image of $N \cap \Stab_G(R)$ in $\Aut(R)^+$ is non-trivial.

In case $R$ is a proper residue of irreducible non-spherical type, we infer by induction on the rank that $\Aut(R)^+$ is simple; notice that the base  of the induction is provided by \cite{Tits:arbres}, which settles the case of trees. Since moreover the homomorphism of $\Stab_G(R)$ in $\Aut(R)^+$ is surjective by Proposition~\ref{prop:extension}, it follows that it remains surjective in restriction to $N \cap \Stab_G(R)$. In other words, we have shown that $\Stab_N(R)$ maps surjectively to $\Aut(R)^+$ for any proper irreducible non-spherical residue. 

Assume now that $R$ is spherical. Thus $R$ is of rank one. Since $(W, I)$ is irreducible, it follows that $R$ is incident with a non-spherical residue $R'$ of rank two. From the part of the claim which has already been proven, we deduce that $\Stab_N(R')$ maps surjectively to $\Aut(R')^+$. Notice that $R'$, viewed as a building in its own right, is a semi-regular tree, in which the residue $R$ corresponds to the set of edges emanating from a fixed vertex. It follows that the canonical map $\Stab_{\Aut(R')^+}(R) \to \Aut(R) = \Aut(R)^+$ is surjective. Therefore, so is the map $\Stab_N(R) \to  \Aut(R) = \Aut(R)^+$. The claim stands proven.

\begin{claim}\label{cl:2}
For any $i \in I$ and any residue $R$ of type $i \cup i^\perp$, the group $\Fix_G(R)$ is contained in $N$. 
\end{claim}

This was established in Lemma~\ref{lem:normal}. 

\begin{claim}\label{cl:3}
Let $J = J_0\cup  J_1 \cup  \dots \cup J_s \subsetneq I$ be the disjoint union of pairwise commuting subsets such that $(W_{J_i}, J_i)$ is irreducible non-spherical for all $i>0$ and $(W_{J_0}, J_0)$ is spherical (and possibly reducible or trivial). Let $c \in \Ch(X)$ and $R = \Res_J(c)$ be its $J$-residue. 

If $\Fix_G(R)$ is contained in $N$, then so is $\Fix_G(\Res_{J_0}(c))$. 
\end{claim}

Set $P = \Stab_G(R)$ and  $U = \Fix_{G}(R)$.  By Proposition~\ref{prop:extension}, the quotient $P/U$ is isomorphic to $\Aut(R)^+$.  

For each $i = 0, \dots, s$, set $R_i = \Res_{J_i}(c)$. By Lemma~\ref{lem:ProductRes}, we have a canonical decomposition $\Ch(R) \cong \Ch(R_0) \times \dots \times \Ch(R_s)$, which  induces a corresponding product decomposition $\Aut(R)^+ \cong L_0 \times \dots \times L_s$, where $L_i = \Aut(R_i)^+$. Let $N'$ denote the image of $N$ in $\Aut(R)^+ \cong L = L_0 \times \dots \times L_s$ under the quotient map $P \to P/U$. Let also $\widetilde{L_j}$ denote the image of $L_j$ under the canonical embedding $L_j \to L$. 

Let $j>0$. Since $N'$ and $\widetilde L_j$ are both normal in $L$, we have $[N', \widetilde{L_j}] \leq N' \cap \widetilde{L_j}$. On the other hand, by Claim~\ref{cl:1}, the group $\Stab_N(R_j)$ maps surjectively to $L_j$. It follows that the  projection $\pi_j \colon L \to L_j$ remains surjective in restriction to $N'$. Therefore, we have 
$$
[L_j, L_j] = [\pi_j(N'), \pi_j(\widetilde{L_j})] = \pi_j( [N', \widetilde{L_j}]) \leq \pi_j(N' \cap \widetilde{L_j}) \leq \pi_j( \widetilde{L_j}) = L_j.
$$ 
Since $R_j$ is of non-spherical type, we know that $L_j$ is simple by induction on the rank, whence $L_j = [L_j, L_j]$ and $N' \cap \widetilde{L_j} = \widetilde{L_j}$. In other words, we have $ \widetilde{L_j} \leq N'$. This holds for all $j>0$; therefore $\{1\} \times L_1 \times \dots \times L_s$ is also contained in $N'$. 

 Recalling that $P$ fits in the short exact sequence 
$$ 1 \to U \to P \to  L_0 \times \dots \times L_s \to 1$$
and that $N$ contains $U$ by hypothesis, we deduce that $N$ contains the preimage of $\{1\} \times L_1 \times \dots \times L_s$ in $P$. This implies the claim, since the group 
$$\Fix_G(R_0) = \Ker(\Stab_G(R_0) \to \Aut(R_0)^+) \leq P$$
coincides with the preimage in $P$ of  $\{1\} \times \Stab_{L_1}(c) \times \dots \times \Stab_{L_s}(c)$.

\begin{claim}\label{cl:4}
$N$ contains the full stabiliser $\Stab_G(R)$ of some proper residue $R$. 
\end{claim}

Since $(W, I)$ is irreducible non-spherical, we have $i \cup i^\perp \subsetneq I$ for all $i \in I$. 
From Claims~\ref{cl:2} and~\ref{cl:3}, we deduce that there exist spherical residues $R_0$ such that $\Fix_G(R_0)$ is contained in $N$. Amongst all such residues, we pick one, say $R$, whose type $J \subseteq I$ is of minimal possible cardinality. 

If $J = \varnothing$, then $R$ is   a single chamber. Thus $\Stab_G(R) = \Fix_G(R)$ is contained in $N$ and we are done. 

Assume next that $J$ is not empty and let $j \in J$. Since $(W, I)$ is irreducible, there exists $i \in I - J$ such that $m_{i, j} = \infty$. 
Now we  distinguish two cases. 

\medskip
Assume first that $J \cup\{i\}$ is properly contained in $I$. Let $R_i$  be the unique residue of type $J \cup \{i\}$ incident with $R$. Then we have $N \geq \Fix_G(R) \geq \Fix_G(R_i)$. Let $R_i = R_0 \times Q_1 \times \dots \times Q_s$ be the decomposition of $R_i$ into a maximal spherical factor $R_0$ and a number of irreducible non-spherical factors. By Claim~\ref{cl:3}, we have $\Fix_G(R_0) \leq N$. By construction $R_i$ is not spherical and is incident to $R$. Therefore the type of $R_0$ is a proper subset of $J$. This contradicts the minimality property of $R$, hence the present case does not occur. 

\medskip
Assume finally that $I = J \cup \{i\}$. Since $(W, I)$ is irreducible, it follows that $m_{i, j'} = \infty$ for all $j' \in J$. In other words, we have $i^\perp = \varnothing$. Therefore, by Claim~\ref{cl:2} we have $\Fix_G(S) \leq N$ for any $i$-residue $S$. It follows from the minimality assumption on $R$ that $J$ has cardinality $1$ as well. Thus $I = \{i, j\}$ and $X$ is a tree, in which case the claim follows from the simplicity theorem in \cite{Tits:arbres}. 
\end{proof}

\section{Fixators of spherical residues}

We now turn to fixators of  spherical residues, i.e. residues whose type $J \subseteq I$ generates a finite subgroup of $W$. We restrict ourselves to the case where the ambient building $X$ is  locally finite case. We endow the group $\Aut(X)$ with the compact open topology; the latter coincides with the topology of pointwise convergence on the discrete set $\Ch(X)$. The group $\Aut(X)$ is locally compact and totally disconnected.

\begin{prop}\label{prop:Fix}
Let $X$ be a semi-regular, locally finite, right-angled building of type $(W, I)$. Let $R$ be a residue of spherical type $J \subseteq I$. Then we have
$$\Fix_{\Aut(X)^+}(R) = \overline{\la U_i(c) \; | \; c \in \Ch(R), \ i \in I - J \ra }.$$

\end{prop}

Specialising (i) to the case   $J = \varnothing$, we obtain 
$$\Stab_{\Aut(X)^+}(c)  = \overline{\la U_i(c) \; | \;  i \in I  \ra }$$
for any chamber $c \in \Ch(X)$.

\begin{proof}[Proof of Proposition~\ref{prop:Fix}]
Let $G = \Aut(X)^+$. For each $n \geq 0$, we set $G(n) = \Fix_G(B(R, n))$. 

Let $c \in \Ch(X)$ and $i \in I$. Set $\overline \sigma = \Res_{i \cup i^\perp}(c)$ and $R' = \proj_{\overline \sigma}(R)$. We say that the pair $(c, i)$ is \textbf{admissible} of $c \in \Ch(R')$ and $\Ch(R') \subseteq X_i(c)$. Now we set 
$$
U(n) =  \la U_i(c) \; | \; (c, i) \text{ is admissible and } \dist(c,R)= n  \ra.
$$
Notice that if $c \in \Ch(R)$, then $(c, i)$ is admissible if and only if $i \not \in J$. Indeed, since $c \in \Ch(R)$, we have $\Ch(R') = \Ch(R) \cap \Ch(\overline \sigma)$; in particular $c \in \Ch(R')$. Now, if $i \in J$, then $J \subset i \cup i^\perp$. Therefore $\Ch(\Res_i(c)) \subseteq \Ch(R) \subset \Ch(\overline \sigma)$ and $R= R'$; in particular $\Ch(R') \not \subseteq X_i(c)$. Conversely, if $  i \not \in J$, then $\Ch(R') \subseteq \Ch(R) \subseteq X_i(c)$ by Corollary~\ref{cor:ResidueWing}, so that $(c, i)$ is indeed admissible. 

This shows that $U(0) = \la U_i(c) \; | \; c \in \Ch(R), \ i \in I - J \ra $. We need to show that $G(0) = \overline{U(0)}$. This is the last of the following series of claims.

\setcounter{claim}{0}

\begin{claim}\label{cl:Fix1}
For all $n \geq 0$, we have $\overline{U(n)} \leq G(n)$. 
\end{claim}

Indeed, let $(c, i)$ be an admissible pair with $\dist(c, R) = n$. Then $B(R, n) \subseteq X_i(c)$ by Lemma~\ref{lem:balls:bis}(i). Thus $U_i(c)$ fixes $B(R, n)$ pointwise, and hence $U(n) \leq G(n)$. The claim follows since $G(n)$ is closed.  

\begin{claim}\label{cl:Fix2}
For all $n \geq 0$, we have $U(n) \leq U(0)$.
\end{claim}

Let $(c, i)$ be an admissible pair with $\dist(c, R) = n$. We prove by induction on $n$ that $U_i(c) \leq U(0)$. The base case $n=0$ is clear; we assume henceforth that $n>0$. Let $x = \proj_R(c)$.  Let $c'$ be the first chamber on a minimal gallery from $c$ to $x$, and let $j \in I$ be the type of the panel $\sigma'$ shared by $c$ and $c'$.

Remark that $\proj_R(c) = \proj_R(c') = x$. Therefore $\proj_R(\sigma') =x $ and it follows from Lemma~\ref{lem:NestedProj} that no panel of $R$ is parallel to $\sigma'$. Setting $R'' = \proj_{\overline{\sigma '}}(R)$, we deduce from Lemmas~\ref{lem:ParallelProj} and Corollary~\ref{cor:ParallelEquivRel} that no panel of $R''$ is parallel to $\sigma'$. Therefore, for any $c'' \in \Ch(R'')$, we have $\Ch(R'') \subseteq X_j(c'')$. It follows that the pair $(c'', j)$ is admissible. Moreover, we have $\dist(c'', R) = \dist(\overline{\sigma'}, R) \leq \dist(c', R) = n-1$. By induction, we have $U_j(c'') \leq U(0)$. 

If $m_{i, j} = 2$, then $j \in i^\perp$ and, by the definition of $j$, we have $c' \in \Res_{i \cup i^\perp}(c)$. But $c'$ is closer to $x$ than $c$. Therefore 
$$
n> \dist(R, \Res_{i \cup i^\perp}(c))= \dist(R, \proj_{\Res_{i \cup i^\perp}(c)}(R))= \dist(R, z)
$$ 
for all $z \in  \Ch(\proj_{\Res_{i \cup i^\perp}(c)}(R))$ by Lemma~\ref{lem:ParallelProj} and Lemma~\ref{lem:ConstantDistance}. 
Therefore we have $c \not \in \Ch(\proj_{\Res_{i \cup i^\perp}(c)}(R))$, in contradiction with the admissibility of $(c, i)$.  Thus $m_{i, j} = \infty$, and hence we have $X_j(c') \subset X_i(c)$ by Lemma~\ref{lem:inclusion}. Since $X_j(c') = X_j(c'')$ by Lemma~\ref{lem:BasicWings}(ii), we conclude that $U_i(c) \leq U_j(c'') \leq U_0)$.

\begin{claim}\label{cl:Fix3}
For all $n \geq 0$, we have $G(n) \leq U(n) G(n+1)$.
\end{claim}

Let $h \in G(n)$. By the local finiteness of $X$, there are only finitely many panels $\sigma_1, \dots, \sigma_r$ with $\Ch(\sigma_s) \subset  B(R, n+1)$ and that are not pointwise fixed by  $h$. The set of panels $\sigma_1, \dots, \sigma_r$  is partitioned according to the relation of parallelism. Upon reordering, we may assume $\{\sigma_1, \dots, \sigma_t\}$ is a set of representatives of those classes such that for all $s < s' \leq t $, the panels $\sigma_s$ and $\sigma_{s'}$ are not parallel. 

Let $i_s \in I$ be the type of $\sigma_s$. It follows from Proposition~\ref{prop:CriterionParallelRAB} that the pairs $(\overline{\sigma_1}, i_1), \dots,  (\overline{\sigma_s}, i_s)$ are pairwise distinct. 

The projection $\proj_{\sigma_s}(R)$ must be a single chamber, say $c_s$, since $h$ fixes $\Ch(R)$ pointwise but acts non-trivially on $\Ch(\sigma_s)$. In particular we have $\dist(c_s, R) = n$. 

Let now $R'_s = \proj_{\overline{\sigma_s}}(R)$ and pick $z \in \Ch(R'_s)$. Since $h$ fixes pointwise $B(R, n)$, it must also fix pointwise $R'_s$. Since $\Res_{i_s}(z)$ is parallel to $\sigma_s$ by Proposition~\ref{prop:CriterionParallelRAB}(ii), it follows that $h$ does not act trivially on $\Ch( \Res_{i_s}(z))$. Therefore 
$$
 \dist(\overline{\sigma_s}, R) = \dist(R'_s, R) = \dist(z, R) \geq n = \dist(c_s, R) \geq  \dist(\overline{\sigma_s}, R).
$$
This implies that $c_s \in \Ch(R'_s)$. Moreover we have $\Ch(R'_s) \subseteq  X_{i_s}(c_s)$, since otherwise $\Ch(\sigma_s)$ would be contained in $\Ch(R'_s)$ by convexity, contradicting that $h$ acts trivially on $\Ch(R'_s)$. Thus the pair $(c_s, i_s)$ is admissible. 

Now it follows from  Lemma~\ref{lem:U:residue} that there is $g \in U(n)$ such that $gh $ fixes $\sigma_s$ pointwise for all  $s = 1, \dots, t$. In particular  $gh$ fixes $\sigma_s$ pointwise for all  $s = 1, \dots, r$. 

By definition $h$ fixes all chambers of $B(x, n+1) - \bigcup_{s=1}^r \Ch(\sigma_s)$. Moreover Lemma~\ref{lem:U} ensures that $g$  fixes all chambers of $B(x, n+1) - \bigcup_{s=1}^t \bigcup_{z \in \Ch(R'_s)} \Ch(\Res_{i_s}(z))$. Let $s \in \{1, \dots, t\}$ and $z \in \Ch(R'_s)$. By Lemma~\ref{lem:ParallelProj} and Lemma~\ref{lem:ConstantDistance}, we have $\dist(z, R) = \dist(R'_s, R) = \dist(c_s, R) = n$. The panels $\sigma_s$ and $\Res_{i_s}(z)$ are parallel by Proposition~\ref{prop:CriterionParallelRAB}(ii). Thus $h$ does not act trivially on $\Res_{i_s}(z)$ and it follows that $\Res_{i_s}(z) \in \{\sigma_1, \dots, \sigma_r\}$. 
This implies that $g$ fixes all chambers of $B(x, n+1) - \bigcup_{s=1}^r \Ch(\sigma_s)$. Hence so does $gh$,   so that $gh \in G(n+1)$ in view of the preceding paragraph. This proves the claim. 

\begin{claim}\label{cl:Fix4}
 $G(0) = \overline{U(0)}$.
\end{claim}
Let $g \in G(0)$. Invoking  Claim~\ref{cl:Fix3} by induction on $n \geq 0$, we find  $u_n \in U(n)$ and $g_n \in G(n+1)$ such that $g = u_0 u_1 \dots u_n g_n$ for all $n$. By Claim~\ref{cl:Fix2} we have $u_n \in U(0)$ for all $n$. Since $\lim_{n \to \infty} g_n = 1$, we obtain $g \in \overline{U(0)}$. This proves that $G(0) \leq \overline{U(0)}$. The reverse inclusion is provided by Claim~\ref{cl:Fix1}.
\end{proof}

\section{Ends and local splittings}

A \textbf{locally normal} subgroup of a locally compact group is a compact subgroup whose normaliser is open. We first record that the automorphism groups of right-angled buildings always admit many locally normal subgroups. 

\begin{lem}\label{lem:LocNorm}
Let $X$ be a thick, semi-regular, locally finite, right-angled building of type $(W, I)$.   Assume that $(W, I)$ is irreducible non-spherical. 

Then $\Aut(X)^+$ admits locally normal subgroups which decompose non-trivially as direct products, all of whose factors are themselves locally normal. 
\end{lem}
\begin{proof}
Given $c \in \Ch(C)$ and $i \in I$, the group $V_i(c)$ is closed by definition, compact because it fixes $c$, and non-trivial by Lemma~\ref{lem:NonAb}. Let $U =\Fix_G(\Res_i(c))$. Since $X$ is locally finite, the group $U$ is a finite intersection of chamber stabilisers, and is thus open in $G$. Moreover, it normalises $V_i(c)$, which proves that   $V_i(c)$ is a locally normal subgroup.   The desired conclusion is thus provided by Proposition~\ref{prop:ProductDec}. 
\end{proof}

The following result is an extended version of Theorem~\ref{thm:product} from the introduction. 

\begin{thm}\label{thm:product:bis}
Let $X$ be a thick, semi-regular, locally finite, right-angled building of type $(W, I)$.   Assume that $(W, I)$ is irreducible non-spherical. 

Then the following are equivalent. 
\begin{enumerate}[(i)]
\item $W$ is one-ended. 

\item $W$ does not split as a free amalgamated product over a finite subgroup. 

\item There is no partition $I = I_0 \cup I_1 \cup I_2$ with $I_1, I_2$ non-empty, $m_{i, j} =2 $ for all $i, j \in I_0$ and $m_{i, j} = \infty$ for all $i \in I_1$ and $j \in I_2$.  
\label{it:partition}

\item $X$ is one-ended.
\label{it:Xone-end}

\item $G$ is one-ended. 
\label{it:Gone-end}

\item All compact open subgroups of $G = \Aut(X)^+$ are indecomposable.
\label{it:LocalFactor}
\end{enumerate}

\end{thm}

We shall need the following basic fact on right-angled Coxeter groups. 

\begin{lem}\label{lem:DeepCorner}
Let $(W, I)$ be an irreducible non-spherical right-angled Coxeter system. 

For any two half-spaces $H, H'$ whose boundary walls cross in the Davis complex of $W$, there is a half-space $H''$ properly contained in $H \cap H'$.
\end{lem}

\begin{proof}
The Davis complex of a right-angled Coxeter group $(W, I)$ is a CAT($0$) cube complex. We call it $\Sigma$. If $(W, I)$ is irreducible, then for every wall $\mathcal W$, we may find a wall $\mathcal W'$ disjoint from $\mathcal W$ (see \cite[Prop.~8.1 p.~309]{Hee}). This implies that $\Sigma$ is irreducible as a cube complex. Moreover, transforming $\mathcal W$ under the dihedral group generated by the reflections through $\mathcal W$ and $\mathcal W'$, we find walls arbitrarily far from $\mathcal W$ in both of the half-spaces that it determines. This proves that $W$ acts essentially on $\Sigma$ in the sense of \cite{CaSe}. Moreover, since $\Sigma$ is irreducible, it follows from  \cite[Th.~4.7]{CaSe} that $W$ does not fix any point at infinity of $\Sigma$. The hypotheses of \cite[Lem.~5.2]{CaSe} are thus fulfilled. The latter result ensures that at least one of the four sectors determined by the boundary walls of $H$ and $ H'$ properly contains a half-space. Transforming that half-space by an appropriate element from the group generated by the reflections fixing $H$ and $H'$, we find a half-space properly contained in $H \cap H'$, as desired.
\end{proof}

We also record an abstract group theoretic fact, where $[g, V]$ denotes the set of commutators $\{[g, v] \; | \; v \in V\}$. 

\begin{lem}\label{lem:abelian}
Let $C$ be a set and $G \leq \Sym(C)$ be a group of permutations of $C$. Let $V \leq G$ be a subgroup fixing all elements of $C$ outside of a subset $Y \subseteq C$. Let $a, b \in G$ such that $Y \cap a(Y) = \varnothing = Y \cap b(Y)$. If each element of $[a, V]$ commutes with each element of $[b, V]$, then $V$ is abelian. 
\end{lem}
\begin{proof}
Given $g \in  \Stab_G(Y)$, we define $\varphi(g) \in \Sym(C)$ by
$$\varphi(g) \colon x \mapsto \left \{
\begin{array}{ll}
g(x) & \text{if } x \in Y,\\
x & \text{otherwise}.
\end{array} \right .
$$
Then  $\varphi \colon \Stab_G(Y) \to \Sym(C)$ is a homormorphism. Since $V$ and $aVa\inv$ have disjoint supports, they are both contained in $\Stab_G(Y)$. Moreover, given $v \in V$, we have
$$\varphi([a, v]) = \varphi(ava\inv v\inv) = \varphi(ava\inv) \varphi(v\inv) = \varphi(v\inv) = v\inv.$$
Similarly $\varphi([b, w])=w\inv$ for all $w  \in V$. Since $[a, v]$ and $[b, w]$ commute by hypothesis, so do their images under $\varphi$. Thus $V$ is abelian, as claimed.
\end{proof}

\begin{proof}[Proof of Theorem~\ref{thm:product:bis}]
The equivalences (i) $\Leftrightarrow$ (ii)  $\Leftrightarrow$ (iii) are well-known, see \cite{MT}.  The equivalence (\ref{it:Xone-end}) $\Leftrightarrow$ (\ref{it:Gone-end}) is clear since $G$ acts properly and cocompactly on $X$, so that $G$ and $X$ are quasi-isometric. 

\medskip \noindent
(i) $\Rightarrow$ (\ref{it:Xone-end}) By assumption all apartments are one-ended. Given $x \in \Ch(X)$, we need to prove that for all $n \geq 0$, any two chambers $c'$, $c''$ at distance~$>n$ away from $x$ can be connected by a gallery  avoiding the ball $B(x, n)$. We proceed by induction on $n$. 

In the base case $n=0$, either a minimal gallery from $c'$ to $c''$ does not pass through $x$, and we are done, or every apartment containing $c'$ and $c''$ also contains $x$, in which case we can find a gallery from $c'$ to $c''$ avoiding $x$ inside one of these apartments, since these are one-ended by hypothesis. 

Let now $n>0$ and assume that $\Ch(X) - B(x, n-1)$ is gallery-connected. Let $c' =c_0, c_1, \dots, c_t = c''$ be a gallery from $c'$ to $c''$ which does not meet $B(x, n-1)$. Then for all $i$, if $c_i \in B(x, n)$ then $\dist(c_{i-1}, x) = \dist(c_{i+1}, x) = n+1$. Therefore, it suffices to prove that if $\dist(c', x) = \dist(c'', x)=n+1$ and $c'$, $c''$ are both adjacent to a common chamber $d \in B(x,n)$, then there is a gallery from $c'$ to $c''$ avoiding $B(x, n)$. Let $\Sigma$ be an apartment containing $x$ and $d$. Let $d'$ and $d''$ be the two chambers of $\Sigma$ different from $d$ and respectively sharing with $d$ the common panel of $d$ and $c'$, and of $d$ and $c''$. Let $i'$ (resp. $i''$) be the type of the panel $\sigma'$ (resp. $\sigma''$) shared by $d, d'$ and $c'$ (resp. $d, d''$ and $c''$). Clearly $\proj_{\overline{ \sigma'}}(d) = d$ and $\proj_{\overline{ \sigma''}}(d) = d$. Therefore, by Lemma~\ref{lem:U}, there is an element $g \in G$ fixing $X_{i'}(d) \cap X_{i''}(d)$ pointwise  and such that $g(d') = c'$ and $g(d'')=c''$. Since $x \in X_{i'}(d) \cap X_{i''}(d)$, it follows that  $ g(\Sigma)$ is an apartment containing $x, d, c'$ and $c''$. Since apartments are one-ended, a gallery joining $c'$ to $c''$ and avoiding $B(x, n)$ can be found in the apartment $g(\Sigma)$, and we are done.

\medskip \noindent
(\ref{it:Gone-end}) $\Rightarrow$ (\ref{it:partition}) Assume that (\ref{it:partition}) fails and let  $I = I_0 \cup I_1 \cup I_2$ be a partition with $I_1, I_2$ non-empty, $m_{i, j} =2 $ for all $i, j \in I_0$ and $m_{i, j} = \infty$ for all $i \in I_1$ and $j \in I_2$. Let $T$ be the graph whose vertex set is the collection of residues of type $I_0 \cup I_1$ and $I_0 \cup I_2$, and declare that two residues are adjacent if they contain a common residue of type $I_0$. By Lemma~4.3 from \cite{HP} the graph $T$ is a tree.  Since $\la I_0 \ra$ is finite and since $X$ is locally finite, the residues of type $I_0$ are finite and, hence, their stabilisers are compact open subgroups. In other words the edge stabilisers of the tree $T$ are compact open subgroups. Since $G$ is chamber-transitive, it acts edge-transitively on $T$. This yields a non-trivial decomposition of $G$ as an amalgamated free product over a compact open subgroup. Hence $G$ cannot be one-ended by \cite{Abels}. 

\medskip \noindent
(\ref{it:LocalFactor})  $\Rightarrow$ (\ref{it:partition}) Assume that (\ref{it:partition}) fails and let  $I = I_0 \cup I_1 \cup I_2$ be a partition with $I_1, I_2$ non-empty, $m_{i, j} =2 $ for all $i, j \in I_0$ and $m_{i, j} = \infty$ for all $i \in I_1$ and $j \in I_2$. Let $R$ be a residue of type $I_0$ in $X$. Since $\la I_0 \ra$ is finite, the set $\Ch(R)$ is finite and hence $\Stab_G(R)$ and $\Fix_G(R)$ are both compact open subgroups of $G$. We shall prove that $\Fix_G(R)$ splits non-trivially as a direct product. 

For $k = 1, 2$, let $U_k = \overline{\la U_i(c) \; | \; c \in \Ch(R), i \in I_k\ra}$. Notice that $U_1$ and $U_2$ are both non-trivial by Lemma~\ref{lem:NonAb} since $I_1$ and $I_2$ are assumed non-empty. 

We claim that $U_1$ and $U_2$ commute. Indeed, let $c_1, c_2 \in \Ch(R)$, let $i_1 \in I_1$, $i_2 \in I_2$. It suffices to prove that $U_{i_1}(c_1)$ and $U_{i_2}(c_2)$ commute. This in turn will follow if one shows that they have disjoint supports. 

By definition the support of $U_{i_1}(c_1)$ is the union of the sets $X_{i_1}(d)$ over all chambers $d$ that are $i_1$-adjacent to but different from $c_1$. Let $d$ be such a chamber. We claim that $X_{i_1}(d) \subset X_{i_2}(c_2)$. 

By Corollary~\ref{cor:ResidueWing}, we have $c_2 \in X_{i_1}(c_1)$ so that $ c_2  \not \in X_{i_1}(d)$. Similarly, Corollary~\ref{cor:ResidueWing}, implies that $c_1 \in  X_{i_2}(c_2)$, which implies that $d \in  X_{i_2}(c_2)$, since otherwise a panel of type $i_1$ would be parallel to a panel of type $i_2$ by Lemma~\ref{lem:ParallPanels}, which is impossible by Proposition~\ref{prop:CriterionParallelRAB}(i). This proves that $d \in  X_{i_2}(c_2)$ and $ c_2  \not \in X_{i_1}(d)$. The claim then follows from Lemma~\ref{lem:inclusion}. 

The claim implies that the support of $U_{i_1}(c_1)$ is pointwise fixed by $U_{i_2}(c_2)$. By symmetry, the support of $U_{i_2}(c_2)$ is pointwise fixed by $U_{i_1}(c_1)$, so that $U_{i_1}(c_1)$ and $U_{i_2}(c_2)$ commute, as desired.  This confirms that  $U_1$ and $U_2$ commute. 

By Proposition~\ref{prop:Fix}, we have $\Fix_G(R)  = \overline{\la U_1 \cup U_2 \ra}$. Since $U_1$ and $U_2$ commute, we have $\la U_1 \cup U_2 \ra = U_1 U_2$. Moreover $U_1$ and $U_2$ are compact, since they are both closed subgroups of the compact group $\Stab_G(R)$. Thus the product $U_1 U_2$ is closed, so that $\Fix_G(R)  =U_1 U_2$. In particular  $U_1 \cap U_2 \leq \centra(\Fix_G(R))$. Hence $U_1 \cap U_2$ is contained in the \textbf{quasi-centre} of $G$, i.e. the collection of elements commuting with an open subgroup. By \cite[Theorem~4.8]{BEW} the group $G$ has trivial quasi-centre since $G$ is compactly generated and simple. Thus $\Fix_G(R) \cong U_1 \times U_2$ as desired.  

\medskip \noindent
(\ref{it:Xone-end}) $\Rightarrow$ (\ref{it:LocalFactor}). Assume finally that 
(\ref{it:Xone-end})  holds and let $U \leq G$ be a compact open subgroup with two commuting subgroups $A, B$ such that $U = A B$. We shall prove that $\overline A$ or $\overline B$ is open. Since the closures $\overline A$ and $\overline B$ commute, we infer that $B$ or $A$ is in the quasi-centre of $G$, which is trivial by \cite[Theorem~4.8]{BEW} since $G$ is compactly generated and simple. Thus $U = A$ or $U=B$ and (\ref{it:LocalFactor}) holds. 

Therefore, all we need to show is that a compact open subgroup $U = A B$ is the commuting product of two closed subgroups $A$ and $B$, then $A$ or $B$ is open. To this end, it suffices   to show that $A$ or $B$ is finite. This follows from the last of a series of claims which we shall now prove successively. 

Let $x \in \Ch(X)$. Upon replacing $A$ and $B$ by their respective intersections with the compact open subgroup $\Stab_G(x)$ and then redefining $U$ accordingly, we may assume that  $U$ fixes $x$.  For all $m \geq 0$, we set $G(m) = \Fix_G(B(x, m))$. Since $U$ is open it contains $G(n_0)$ for some $n_0 \geq 0$. Without loss of generality, we may assume that $n_0>1$. We define 
$$\Pi = \{ \sigma \text{ panel of } X \; | \;  \Stab_{G(n_0)}(\sigma) \not \leq \Fix_G(\sigma)\}.$$ 
In particular, if $\sigma \in \Pi$ then $\dist(\sigma, x)\geq n_0$.

Moreover, to each chamber $c \in  \Ch(X)$, we associate two subsets  of $I$ defined as follows:
$$I_0(c) = \{i \in I \; | \; \proj_{\Res_i(c)}(x) \neq c\}$$
and 
$$I_\Pi(c) = \{i \in I \; | \; \Res_i(c) \in \Pi\}.$$
Recall that a subset $J \subseteq I$ is called \textbf{spherical} if it generates a finite subgroup of $W$. It is a classical fact that $I_0(c)$ is a spherical subset of $I$. 

\setcounter{claim}{0}

\begin{claim}\label{cl:I_Pi}
Let $c \in \Ch(X)$ with $\dist(c, x)> n_0$ and $i \in I$. Then $i \in I_\Pi(c)$ if and only if $\dist(x, \Res_{i \cup i^\perp}(c)) \geq n_0$.
\end{claim}

Let $\sigma$ be the $i$-panel of $c$, let ${\overline \sigma}$ be the $(i \cup i^\perp)$-residue of $c$ and $c' = \proj_{\overline \sigma}(x)$.  

If $\dist(x, {\overline \sigma})  = \dist(x, c') \geq n_0$, then $U_i(c')$ fixes $B(x, n_0)$ pointwise by Corollary~\ref{cor:balls}. Thus $U_i(c') \leq G(n_0) \leq U$. Since $\sigma$ is parallel to the $i$-panel of $c'$ by Proposition~\ref{prop:CriterionParallelRAB}(ii), we infer that $U_i(c')$ fixes $\proj_\sigma(x)$ and permutes arbitrarily all the other chambers of $\sigma$ by Proposition~\ref{prop:Panel:extension}. Therefore $\Stab_{G(n_0)}(\sigma) \not \leq \Fix_G(\sigma)$. Thus  $\sigma \in \Pi$ and $i \in I_\Pi(c)$. 

Assume conversely that $\dist(x, {\overline \sigma}) < n_0$. Then the $i$-panel of $c'$ lies entirely in $B(x, n_0)$ and is thus pointwise fixed by $G(n_0)$. Therefore $\Stab_{G(n_0)}(\sigma)$ acts trivially on $\sigma$, and hence $\sigma \not \in \Pi$ and $i \not \in I_\Pi(c)$. 

\begin{claim}\label{cl:I_Pispherical}
There exists $n_1 > n_0$ such that for all $c \in \Ch(X)$ with $\dist(c, x) > n_1$, we have    $I_0(c) \cap I_\Pi(c) \neq \varnothing$. 
\end{claim}


Since $(W, I)$ is right-angled, any collection of pairwise intersecting walls in an apartment is contained in the set of walls of a spherical residue. The cardinality of such a collection is bounded above by the largest cardinality of a spherical subset of $I$. In particular it is finite. In view of Ramsey's theorem, we infer that there is some $n_1> n_0$ such that any set of more than $n_1$ walls contains a subset of more than $n_0+1$ pairwise non-intersecting walls. 

Let now $c \in \Ch(X)$ be such that  $\dist(c, x) > n_1$ and
%
$\Sigma$ be an apartment containing $c$ and $x$. By construction there is a set of more than $n_0+1$ pairwise non-intersecting walls in $\Sigma$ that are crossed by any minimal gallery from $c$ to $x$. In particular, at least one of these walls, say $\mathcal W$, separates $c$ from the ball $B(x, n_0+1)$.

Since $(W, I)$ is right-angled, no wall crossed by a shortest possible gallery from $c$ to a chamber adjacent to $\mathcal W$ crosses $\mathcal W$. Let $\mathcal W'$ be the first wall crossed by such a gallery. Thus $\mathcal W'$ is adjacent to $c$,  and every chamber adjacent to $\mathcal W'$ is at distance $> n_0$ from $x$.

Let now  $k \in I$ be the type of the panel of $c$ which belongs to $\mathcal W'$. Since $\mathcal W'$ separates $c$ from $x$, we have $k \in I_0(c)$. Notice that $\proj_{\Res_{k \cup k^\perp}(c)}(x)$ belongs to $\Sigma$. Thus $\proj_{\Res_{k \cup k^\perp}(c)}(x)$ is a chamber of $\Sigma$ which is adjacent to the wall $\mathcal W'$. This implies that $\dist(x, \Res_{k \cup k^\perp}(c))> n_0$.   Therefore $k \in I_\Pi(c)$ by Claim~\ref{cl:I_Pi}. Thus the sets $I_0(c)$  and $I_\Pi(c)$ have a non-empty intersection, as desired.

\begin{claim}\label{cl:AB}
Let $c \in \Ch(X)$  and $\sigma$ be a panel of $c$. If $a(c) \neq c$ for some $a \in \Stab_A(\sigma)$, then $b(c) = c$ for all $b \in \Stab_B(\sigma)$, and similarly with $A$ and $B$ interchanged.
\end{claim}

Let $i \in I$ be the type of $\sigma$. Notice that $c_0 = \proj_\sigma(x) \neq c$ since $a$ fixes $x$ and stabilises $\sigma$. Let $\Sigma$ be an apartment containing $c$ and $x$. It also contains $c_0$ by convexity. 

Since $(W, I)$ is irreducible and non-spherical,  there is $j \in I$ such that $m_{i, j}=\infty$. Let $R = \Res_{\{i, j\}}(c)$. Let $r$ be the reflection of $\Sigma$ swapping $c$ and $c_0$ and $r'$ be the reflection of $\Sigma$ through the $j$-panel of $c$. We set $c' = (r' r)^{n_0}(c)$ and $c'_0 = (r' r)^{n_0}(c_0)$. Thus $c$ and $c'$ are separated by $2n_0$ walls of the residue $R$ in $\Sigma$, and $x$ lies on the same side as $c$ of all those walls. Set $y = \proj_{\Res_{i \cup i^\perp}(c')}(x)$. Then $y$ belongs to $\Sigma$ since apartments are convex. The chambers $c', c'_0$ and $y$ are all adjacent to the wall $\mathcal W = \Sigma \cap  \Res_{i \cup i^\perp}(c')$. Moreover the chambers $x,y$ and $ c'_0$ lie on the same side of  $\mathcal W$ while $c'$ lies on the opposite side. Thus $X_i(c')$ and $X_i(c'_0)$ are disjoint, and $X_i(c'_0) = X_i(y) \supseteq B(x, n_0)$ by Lemma~\ref{lem:BasicWings}(ii) and Corollary~\ref{cor:balls}.
In particular $X_i(c') \cap  B(x, n_0) = \varnothing$ so that $V_i(c')$ fixes $B(x, n_0)$ pointwise, and is thus contained in $U$.

By construction, we have  $X_i(c') \cap \Sigma \subseteq X_i(c) \cap \Sigma$. Lemma~\ref{lem:inclusion} therefore ensures that $X_i(c') \subset X_i(c)$, whence $V_i(c') \leq V_i(c)$. In particular the support of $V_i(c')$ and its image under $a$ are disjoint.

Similarly, if $b \in \Stab_B(\sigma)$ and $b(c) \neq c$, then  the support of $V_i(c')$ and its image under $b$ are disjoint. Since $[a, V_i(c')] \leq A$ and $[b, V_i(c')] \leq B$, we deduce from Lemma~\ref{lem:abelian} that   $V_i(c')$ is abelian, in contradiction with Lemma~\ref{lem:NonAb}. Therefore $b(c) = c$ for all $b \in  \Stab_B(\sigma)$.

\begin{claim}\label{cl:f}
For each panel $\sigma \in \Pi$, there is a unique $F \in \{A, B\}$ with $\Stab_F(\sigma) \not \leq \Fix_G(\sigma)$. We denote the corresponding function by 
$$f \colon \Pi \to \{A, B\} \colon \sigma \mapsto F.$$
Moreover, the group $\Stab_{F}(\sigma)$ permutes arbitrarily the elements of $\Ch(\sigma)$ different from $\proj_\sigma(x)$ (i.e. it induces the full symmetric group on $\Ch(\sigma) - \{\proj_\sigma(x)\}$).
\end{claim}

Let $\sigma \in \Pi$. By definition there is $u \in \Stab_{G(n_0)}(\sigma)$ and $c \in \Ch(\sigma)$ with $u(c) \neq c$. Write $u = ab$ with $a \in A$ and $b \in B$. Consider a gallery $x_0, x_1, \dots, x_k $  of minimal possible length joining a chamber in $B(x, n_0)$ to a chamber in $\Ch(\sigma)$. Since $u \in G(n_0)$, it fixes $x_0$ and by the minimality of the gallery, we have $x_k = \proj_\sigma(x_0)$, so that $u$ fixes $x_k$ as well. Therefore $u$ fixes $x_i$ for all $i$. 

We claim that $a$ and $b$ both fix $x_i$ for all $i$. Otherwise there is some $i$ such that  $a(x_i) \neq x_i$. Since $u(x_i)=x_i$ and $u =ab$, we must have $b(x_i) \neq x_i$. If $i$ the smallest such index, then $a$ and $b$ also fix $x_{i-1}$ and thus both stabilise the panel shared by $x_{i-1}$ and $x_i$. This contradicts Claim~\ref{cl:AB}. 

It follows that $a$ and $b$ both fix $x_k$ and hence stabilise $\sigma$. As $ab = u \not \in \Fix_G(\sigma)$,   we have $\Stab_A(\sigma) \not \leq \Fix_G(\sigma)$ or $\Stab_B(\sigma) \not \leq \Fix_G(\sigma)$. It remains to show that these two possibilities are mutually exclusive. Let $\Ch_A$ and $\Ch_B$ be the subsets of $\Ch(\sigma)$ that are not fixed by $A$ and $B$ respectively. The previous claim guarantees that $\Ch_A$ and $\Ch_B$ are disjoint, and thus they are both stabilised by $A$ and $B$ and hence by $U$. 

Let $i \in I$ be the type of $\sigma$ and $c' = \proj_\sigma(x)$. Since $\sigma \in \Pi$, we have $i \in I_\Pi$ and $U_i(c') \leq G(n_0) \leq U$ by Claim~\ref{cl:I_Pi} and Corollary~\ref{cor:balls}. Consequently the group $\Stab_U(\sigma)$ permutes arbitrarily the set $\Ch(\sigma) - \{c'\}$ by Proposition~\ref{prop:Panel:extension}. Since $\Ch_A$ and $\Ch_B$ are disjoint and $U$-invariant, it follows that either $\Ch_A$ or $\Ch_B$ coincides with the whole of $\Ch(\sigma) - \{c'\}$.

\begin{claim}\label{cl:ij=2}
Let $c \in \Ch(X)$ and $i, j \in I$ with $m_{i, j} = 2$. Let $\sigma_i$ and $\sigma_j$ be the $i$- and $j$-panels of $c$ respectively. If $\sigma_i$ and $\sigma_j$ belong to $\Pi$, then $f(\sigma_i) = f(\sigma_j)$. 
\end{claim}

Suppose for a contradiction that $f(\sigma_i) = A$ and $f(\sigma_j)=B$. Then there exist $a \in A$, $b \in B$ stabilising respectively $\sigma_i$ and $\sigma_j$, and such that  $a(c_i) \neq c_i$ and $b(c_j) \neq c_j$ for some  $c_i \in \Ch(\sigma_i)$ and $c_j \in \Ch(\sigma_j)$.

Let $R$ be the $\{i, j\}$-residue of $c$ and set $c' = \proj_R(x)$. Let also $\sigma'_i$ and $\sigma'_j$ be the  $i$- and $j$-panels of $c'$. Then  $a$ and $b$ both fix $c'$ and  stabilise  $\sigma'_i$ and $\sigma'_j$. Moreover $\sigma'_i$ and $\sigma'_j$ are respectively parallel to $\sigma_i$ and $\sigma_j$. Set $c'_i = \proj_{\sigma'_i}(c_i)$ and $c'_j = \proj_{\sigma'_j}(c_j)$. Then $a(c'_i) \neq c'_i$ and $b(c'_j) \neq c'_j$. 
Therefore we have $f(\sigma_i) = f(\sigma'_i)$ and $f(\sigma_j) = f(\sigma'_j)$. 

Let $\Sigma$ be an apartment containing $x$ and $c'$. By Claim~\ref{cl:I_Pi} and Corollary~\ref{cor:balls}, the ball $B(x, n_0)$ is contained in $X_i(c') \cap X_j(c')$. From Lemma~\ref{lem:U}, we deduce that there is some $g \in G(n_0) \leq U$ mapping $\Sigma$ to an apartment containing $c'_i$ and $c'_j$. Upon replacing $\Sigma$ by $g(\Sigma)$, we can thus assume that $\Sigma$ is an apartment containing the chambers $x, c', c'_i$ and $c'_j$.  

Let $H$ (resp. $H'$) be the half-apartment of $\Sigma$ containing $c'_i$ (resp. $c'_j$) but not $c'$. Since $(W, I)$ is irreducible and non-spherical, there is a half-apartment $H''$ which is entirely contained in $H \cap H'$ by Lemma~\ref{lem:DeepCorner}. Let $c''$ be a chamber of $H''$ having a panel in the wall determined by $H''$, and let $k \in I$ be the type of that panel. Since $H'' \subset H \cap H'$, we deduce from Lemma~\ref{lem:inclusion}   that $X_k(c'') \subseteq X_i(c'_i) \cap X_j(c'_j)$. In particular we have $V_k(c'') \leq V_i(c'_i) \cap V_j(c'_j) \leq G(n_0) \leq U$. 

We see that the support of $V_k(c'')$ and its image under $a$ are disjoint. Similarly, the support of $V_k(c'')$ and its image under $b$ are disjoint. Since   $[a, V_k(c'')] \leq A$ and   $[b, V_k(c'')] \leq B$, we deduce from Lemma~\ref{lem:abelian}  that $V_k(c'')$ is abelian, in contradiction with Lemma~\ref{lem:NonAb}. The claim stands proven.

\begin{claim}\label{cl:ij=infty}
Let $c \in \Ch(X)$ and $i, j \in I$ with $m_{i, j} = \infty$. Let $\sigma_i$ and $\sigma_j$ be the $i$- and $j$-panels of $c$ respectively. If $\sigma_i$ and $\sigma_j$ belong to $\Pi$, and if $ \proj_{\sigma_i}(x) \neq c$, then $f(\sigma_i) = f(\sigma_j)$. 
\end{claim}

Suppose for a contradiction that $f(\sigma_i) = A$ and $f(\sigma_j)=B$ (the case  $f(\sigma_j) =A$ and $f(\sigma_i)=B$ is treated similarly). In view of Claim~\ref{cl:f} and the fact that $c' = \proj_{\sigma_i}(x) \neq c$, we can find $a \in A$, $b \in B$ and $c_j \in \Ch(\sigma_j)$ such that $a(c) \neq c$ and $b(c_j) \neq c_j$. 

By Claim~\ref{cl:I_Pi} and Corollary~\ref{cor:balls}, the ball $B(x, n_0)$ is contained in $X_i(c')$. By Lemma~\ref{lem:inclusion} we have $X_i(c) \supset X_j(c_j)$. In particular $X_j(c_j)$ is disjoint from $B(x, n_0)$, from which it follows that $V_j(c_j)$ is contained in $U$. Therefore  $[a, V_j(c_j)] \subseteq A$ and $[b, V_j(c_j)] \subseteq B$. Since moreover $a$ (resp. $b$) maps the support of $V_j(c_j)$ to a disjoint subset. As before, Lemma~\ref{lem:abelian} then implies that   $V_j(c_j)$ is abelian, in contradiction with Lemma~\ref{lem:NonAb}.

\begin{claim}\label{cl:constant}
Let $c \in \Ch(X)$, let $i, j \in I$ and let  $\sigma_i$ and $\sigma_j$ be the $i$- and $j$-panels of $c$ respectively. If   $\sigma_i$ and $\sigma_j$ belong to $\Pi$, and if $\dist(c, x) > n_1$, then $f(\sigma_i) = f(\sigma_j)$. 
\end{claim}

It suffices to deal with the case when $\proj_{\sigma_i}(x) = \proj_{\sigma_j}(x)=c$, since the other cases are dealt with by Claims~\ref{cl:ij=2} and~\ref{cl:ij=infty}. 



Since $\dist(c, x) > n_1$, there is some $k \in I_0(c) \cap I_\Pi(c)$ by Claim~\ref{cl:I_Pispherical}. Let  $\sigma_k$ be the $k$-panel of $c$. Invoking Claim~\ref{cl:ij=2} or Claim~\ref{cl:ij=infty} according as $m_{i, k} = 2$ or $m_{i, k} = \infty$, we infer that $f(\sigma_i)= f(\sigma_k)$. Similarly $f(\sigma_j)= f(\sigma_k)$, so that $f(\sigma_i)= f(\sigma_j)$ and we are done.


\bigskip
Notice that by Claim~\ref{cl:I_Pispherical}, every chamber $c$ at distance~$>n_1$ from $x$ has a panel belonging to $\Pi$. Moreover the map $f$ takes the same value on all these panels by Claim~\ref{cl:constant}. We shall denote this common value by $f(c)$. 

\begin{claim}
Let $c, c' \in \Ch(X)$ be two adjacent chambers both at distance~$ > n_1$ from $x$. Then $f(c) = f(c')$. 
\end{claim}

Let $\sigma$ be the panel shared by $c$ and $c'$. If $\sigma \in \Pi$ then we are done by the previous claim.  We assume henceforth that $\sigma \not \in \Pi$ and denote by $j $ its type. 
By Claim~\ref{cl:I_Pispherical} there is some $i \in I_0(c) \cap I_\Pi(c)$. Let $\sigma_i$ be the $i$-panel of $c$. Then $d = \proj_{\sigma_i}(x)$ is different from $c$ and moreover $\sigma_i \in \Pi$. By Claim~\ref{cl:I_Pi} and Corollary~\ref{cor:balls},  this implies that $B(x, n_0)$ is entirely contained in $X_i(d)$. It follows that $ m_{i, j}=2$, since otherwise we would have $X_i(d) \subset X_j(c)$ by Lemma~\ref{lem:inclusion} and hence $\dist(x, \Res_{j \cup j^\perp}(c)) \geq n_0$. 
 This would contradict Claim~\ref{cl:I_Pi} since $\sigma \not \in \Pi$.   

Since $m_{i, j}=2$, it follows that the $i$-panel of $c'$, say $\sigma'_i$, is parallel to $\sigma_i$ since they are contained and opposite in the  $\{i, j\}$-residue of $c$. Therefore, any element of $G(n_0) \leq U$ stabilises $\sigma_i$ and acts non-trivially on it if and only if it stabilises $\sigma'_i$ and acts non-trivially on it. Hence   $f(\sigma_i) = f(\sigma'_i)$ and therefore $f(c)=f(c')$.

\begin{claim}
We have $A \cap G(n_1+1) =1$ or $B \cap G(n_1+1)=1$.
\end{claim}

By (\ref{it:Xone-end}) any two chambers at distance~$ > n_1$ from $x$ can be joined by a gallery which does not meet the ball $B(x, n_1)$. By the preceding claim, this implies that the map $f$ is constant on $\Ch(X) -  B(x, n_1)$. Upon exchanging $A$ and $B$ we may assume that this constant value is $A$. It follows that for all panels $\sigma \in \Pi$ at distance~$> n_1$ from $x$, we have $\Stab_B(\sigma) \leq \Fix_B(\sigma)$. An immediate induction now shows that for all $m > n_1$, we have $B \cap G(m) \leq G(m+1)$. Therefore $B \cap G(n_1+1)$ is trivial. 
\end{proof}

\providecommand{\bysame}{\leavevmode\hbox to3em{\hrulefill}\thinspace}
\providecommand{\MR}{\relax\ifhmode\unskip\space\fi MR }
\providecommand{\MRhref}[2]{%
  \href{http://www.ams.org/mathscinet-getitem?mr=#1}{#2}
}
\providecommand{\href}[2]{#2}


\begin{thebibliography}{CRW13b}

\bibitem[AB08]{AB}
Peter Abramenko and Kenneth~S. Brown, \emph{Buildings}, Graduate Texts in
  Mathematics, vol. 248, Springer, New York, 2008, Theory and applications.

\bibitem[Abe74]{Abels}
Herbert Abels, \emph{Specker-{K}ompaktifizierungen von lokal kompakten
  topologischen {G}ruppen}, Math. Z. \textbf{135} (1973/74), 325--361.

\bibitem[BEW11]{BEW}
Yiftach Barnea, Mikhail Ershov, and Thomas Weigel, \emph{Abstract
  commensurators of profinite groups}, Trans. Amer. Math. Soc. \textbf{363}
  (2011), no.~10, 5381--5417.

\bibitem[Bou97]{Bourdon}
Marc Bourdon, \emph{Immeubles hyperboliques, dimension conforme et rigidit\'e
  de {M}ostow}, Geom. Funct. Anal. \textbf{7} (1997), no.~2, 245--268.

\bibitem[CM11]{CM}
Pierre-Emmanuel Caprace and Nicolas Monod, \emph{Decomposing locally compact
  groups into simple pieces}, Math. Proc. Cambridge Philos. Soc. \textbf{150}
  (2011), no.~1, 97--128.

\bibitem[CRW13a]{CRW-short}
Pierre-Emmanuel Caprace, Colin Reid, and George Willis, \emph{Locally normal
  subgroups of simple locally compact groups}, arXiv preprint 1303.6755, 2013.

\bibitem[CRW13b]{CRW}
\bysame, \emph{Locally normal subgroups of totally disconnected groups. {P}art
  {I}: {G}eneral theory}, arXiv preprint 1304.5144, 2013.

\bibitem[CS11]{CaSe}
Pierre-Emmanuel Caprace and Michah Sageev, \emph{Rank rigidity for {CAT}(0)
  cube complexes}, Geom. Funct. Anal. \textbf{21} (2011), no.~4, 851--891.

\bibitem[Dav98]{Davis}
Michael~W. Davis, \emph{Buildings are {${\rm CAT}(0)$}}, Geometry and
  cohomology in group theory ({D}urham, 1994), London Math. Soc. Lecture Note
  Ser., vol. 252, Cambridge Univ. Press, Cambridge, 1998, pp.~108--123.

\bibitem[{Gis}09]{Gis}
Jakub {Gismatullin}, \emph{{Boundedly simple groups of automorphisms of
  trees}}, Preprint, arXiv:0905.0913, 2009.

\bibitem[H{\'e}e93]{Hee}
Jean-Yves H{\'e}e, \emph{Sur la torsion de steinberg--ree des groupes de
  chevalley et des groupes de kac--moody}, Th\`ese d'\'Etat de l'Universit\'e
  Paris 11 Orsay, 1993.

\bibitem[HP98]{HP:simple}
Fr{\'e}d{\'e}ric Haglund and Fr{\'e}d{\'e}ric Paulin, \emph{Simplicit\'e de
  groupes d'automorphismes d'espaces \`a courbure n\'egative}, The {E}pstein
  birthday schrift, Geom. Topol. Monogr., vol.~1, Geom. Topol. Publ., Coventry,
  1998, pp.~181--248 (electronic).

\bibitem[HP03]{HP}
\bysame, \emph{Constructions arborescentes d'immeubles}, Math. Ann.
  \textbf{325} (2003), no.~1, 137--164.

\bibitem[Laz]{Laz}
Nir Lazarovich, \emph{Simplicity of automorphism groups of rank one cube
  complexes}, Preprint, 2012.

\bibitem[MT09]{MT}
Michael Mihalik and Steven Tschantz, \emph{Visual decompositions of {C}oxeter
  groups}, Groups Geom. Dyn. \textbf{3} (2009), no.~1, 173--198.

\bibitem[Ron89]{Ronan}
Mark Ronan, \emph{Lectures on buildings}, Perspectives in Mathematics, vol.~7,
  Academic Press Inc., Boston, MA, 1989.

\bibitem[Tit64]{Tits64}
Jacques Tits, \emph{Algebraic and abstract simple groups}, Ann. of Math. (2)
  \textbf{80} (1964), 313--329.

\bibitem[Tit70]{Tits:arbres}
\bysame, \emph{Sur le groupe des automorphismes d'un arbre}, Essays on topology
  and related topics ({M}\'emoires d\'edi\'es \`a {G}eorges de {R}ham),
  Springer, New York, 1970, pp.~188--211.

\bibitem[Tit74]{TitsLN}
\bysame, \emph{Buildings of spherical type and finite {BN}-pairs}, Lecture
  Notes in Mathematics, Vol. 386, Springer-Verlag, Berlin, 1974.

\end{thebibliography}
\end{document}